\documentclass{jgm}
\usepackage{amsmath}
  \usepackage{paralist}
  \usepackage{graphics} %% add this and next lines if pictures should be in esp format
  \usepackage{epsfig} %For pictures: screened artwork should be set up with an 85 or 100 line screen
 \usepackage[colorlinks=true]{hyperref}
\hypersetup{urlcolor=blue, citecolor=red}

\def\currentvolume{6}
\def\currentissue{3}
 \def\currentyear{2014}
  \def\currentmonth{September}
    \def\ppages{279-296}
     \def\DOI{10.3934/jgm.2014.6.279}

\newtheorem{theorem}{Theorem}[section]
\newtheorem{corollary}{Corollary}

\newtheorem{lemma}[theorem]{Lemma}
\newtheorem{proposition}{Proposition}

\theoremstyle{definition}
\newtheorem{definition}[theorem]{Definition}

\title[Warped Poisson brackets on warped products]
      {Warped Poisson brackets on warped products}

\author[Yacine A\"{\i}t Amrane, Rafik Nasri and Ahmed Zeglaoui]{}


 \email{yacinait@hotmail.com}
 \email{rmag1math@yahoo.fr}
 \email{ahmed.zeglaoui@gmail.com}

\begin{document}
\maketitle

% Enter the first author's name and address:

\centerline{\scshape Yacine A\"{\i}t Amrane}
\medskip
{\footnotesize
% please put the address of the first author
 \centerline{Laboratory of Algebra and Number Theory}
   \centerline{Facult\'e de Math\'ematiques}
   \centerline{USTHB, BP32, El-Alia, 16111 Bab-Ezzouar, Alger, Algeria}
} % Do not forget to end the {\footnotesize by the sign }

\medskip

\centerline{\scshape Rafik Nasri}
\medskip
{\footnotesize
% please put the address of the first author
 \centerline{Laboratory of Geometry, Analysis, Control and Applications}
   \centerline{Universit\'e de Sa\"{i}da}
   \centerline{BP138, En-Nasr, 20000 Sa\"ida, Algeria}
} % Do not forget to end the {\footnotesize by the sign }

\medskip

\centerline{\scshape Ahmed Zeglaoui}
\medskip
{\footnotesize
 % please put the address of the second  and third author
 \centerline{Laboratory of Algebra and Number Theory}
   \centerline{Facult\'e de Math\'ematiques}
   \centerline{USTHB, BP32, El-Alia, 16111 Bab-Ezzouar, Alger, Algeria} }

\bigskip

% The name of the associate editor will be entered by an editorial staff
% "Communicated by the associate editor name" is not needed for special issue.
 \centerline{(Communicated by Manuel de Le\'on)}

%The abstract of your paper
\begin{abstract}
In this paper, we generalize the geometry of the product
pseudo-Riemannian manifold equipped with the product Poisson
structure (\cite{Nas2}) to the geometry of a warped product of
pseudo-Riemannian manifolds equipped with a warped Poisson
structure. We construct three bivector fields on a product manifold
and show that each of them lead under certain conditions to a
Poisson structure. One of these bivector fields will be called the
warped bivector field. For a warped product of pseudo-Riemannian
manifolds equipped with a warped bivector field, we compute the
corresponding contravariant Levi-Civita connection and the
curvatures associated with.
\end{abstract}
%The title of your section 1

\section{Introduction}

Poisson manifolds play a fundamental role in Hamiltonian dynamics,
where they serve as a phase space. The geometry of Poisson
structures, which began as an outgrowth of symplectic geometry, has
seen rapid growth in the last three decades, and has now become a
very large theory, with interaction with many domains of
mathematics, including Hamiltonian dynamics, integrable systems,
representation theory, quantum groups, noncommutative geometry,
singularity theory \ldots

The warped product provides a way to construct new pseudo-Riemannian
manifolds from given ones, see \cite{ONeil},\cite{bishop} and
\cite{Beem2}. This construction has useful applications in general
relativity, in the study of cosmological models and black holes. It
generalizes the direct product in the class of pseudo-Riemannian
manifolds and it is defined as follows. Let $(M_1,\tilde{g}_1)$ and
$(M_2,\tilde{g}_2)$ be two pseudo-Riemannian manifolds and let
$f:M_1\longrightarrow\mathbb{R}$ be a positive smooth function on
$M_1$, the warped product of $(M_1,\tilde{g}_1)$ and
$(M_2,\tilde{g}_2)$ is the product manifold $M_1\times M_2$ equipped
with the metric tensor $\tilde{g}_{f}:=\sigma_1^*\tilde{g}_1
+(f\circ \sigma_1)\sigma_2^*\tilde{g}_2$, where $\sigma_1$ and
$\sigma_2$ are the projections of $M_1\times M_2$ onto $M_1$ and
$M_2$ respectively. The manifold $M_1$ is called the base of
$(M_1\times M_2,\tilde{g}_f)$ and $M_2$ the fiber. The function $f$
is called the warping function.

This paper extends the study of product pseudo-Riemannian manifolds
equipped with product Poisson structures, \cite{Nas2}, to warped
products equipped with warped Poisson structures. We will give an
expression for the contravariant Levi-Civita connection and for its
curvatures.

This paper is organized as follows. In section 2, we recall basic
definitions and facts about contravariant connections, generalized
Poisson brackets and pseudo-Riemannian Poisson manifolds. By analogy
with the definition of the Hessian and the gradient of a smooth
function in the covariant case, in section 3 we define some
differential operators associated to the Levi-Civita contravariant
connection. In section 4, we consider bivector fields $\Pi_1$ and
$\Pi_2$ on manifolds $M_1$ and $M_2$ respectively and, for a smooth
function $\mu$ on $M_1$, we define the warped bivector field
$\Pi^{^\mu}$ on $M_1\times M_2$ relative to $\Pi_1,\Pi_2$ and the
warping function $\mu$, then, we give conditions under which
$\Pi^{^\mu}$ becomes a warped Poisson tensor; next, if
$\mathcal{D}^1$ and $\mathcal{D}^2$ are contravariant connections
with respect to $\Pi_1$ and $\Pi_2$, we define a contravariant
connection $\mathcal{D}^\mu$ with respect to the bivector field
$\Pi^{\mu}$ as the warped product of $\mathcal{D}^1$ and
$\mathcal{D}^2$ using the warping function $\mu$, then, we express
the generalized pre-Poisson bracket associated with
$\mathcal{D}^\mu$ in terms of the generalized pre-Poisson brackets
$\{.,.\}_1$ and $\{.,.\}_2$ associated with $\mathcal{D}^1$ and
$\mathcal{D}^2$ respectively, it will be clear that $\{.,.\}$ is a
generalized Poisson bracket, i.e. satisfy the graded Jacobi
identity, if and only if $\{.,.\}_1$ and $\{.,.\}_2$ are; to end
this section, we give two other bivector fields on $M_1\times M_2$
and conditions under which they become Poisson tensors. In the final
section 5, we consider bivector fields $\Pi_1$ and $\Pi_2$ on
pseudo-Riemannian manifolds $(M_1,\tilde{g}_1)$ and
$(M_2,\tilde{g}_2)$ respectively, we compute the Levi-Civita
contravariant connection $\mathcal{D}$ associated with the pair
$(g^f,\Pi^{^\mu})$, where $g^f$ is the cometric of the warped metric
$\tilde{g}_{\frac{1}{f}}$, and we compute the curvatures of
$\mathcal{D}$; finally, if $g_1$ and $g_2$ denote the cometrics of
$\tilde{g}_1$ and $\tilde{g}_2$ respectively, we conclude with some
interesting relationships between the geometry of the triples
$(M_1,g_1,\Pi_1)$ and $(M_2,g_2,\Pi_2)$ and that of $(M_1\times
M_2,g^f,\Pi^{^\mu})$.

\section{Preliminaries}

\subsection{Poisson structures}

Many fundamental definitions and results about Poisson manifolds can
be found in \cite{Vais} and \cite{Jean-Paul Dufour}.

A Poisson bracket on a manifold $M$ is a Lie bracket $\{.,.\}$ on
$C^{\infty}(M)$ satisfying the Leibniz identity
$$
\{\varphi, \phi \psi\} = \{\varphi, \phi\}\psi + \phi\{\varphi,
\psi\}, \quad \forall \; \varphi, \phi, \psi\in C{^\infty}(M).
$$
A Poisson manifold is a manifold equipped with a Poisson bracket.
The Leibniz identity means that, for a given function $\varphi\in
C^\infty(M) $ on a Poisson manifold $M$, the map $\psi\mapsto
\{\varphi,\psi\}$ is a derivation, and thus, there is a unique
vector field $X_\varphi$ on $M$, called the Hamiltonian vector field
of $\varphi$, such that for any $\psi\in C^\infty(M)$ we have
$$
X_\varphi(\psi)=\{\varphi,\psi\}.
$$
A function $\varphi\in C^\infty(M)$ is called a Casimir function if
$X_\varphi\equiv 0$. It follows from the Leibniz identity that there
exists a bivector field $\Pi\in\Gamma(\wedge^2TM)$ such that
\begin{equation}\label{prePoisson}
\{\varphi, \psi\} = \Pi(d\varphi, d\psi).
\end{equation}
If we denote by $[.,.]_S$ the Schouten-Nijenhuis bracket, we have
$$
\{\{\varphi,\phi\},\psi\}+\{\{\phi,\psi\},\varphi\}+\{\{\psi,\varphi\},\phi\}
=\frac{1}{2}[\Pi,\Pi]_S(d\varphi,d\phi,d\psi).
$$
Therefore, the Jacobi identity for $\{., .\}$ is equivalent to the
condition $[\Pi, \Pi]_S = 0$. Conversely, if $\Pi$ is a bivector
field on a manifold $M$ such that $[\Pi,\Pi]_S=0$, then the bracket
$\{.,.\}$ defined on $C^\infty(M)$ by (\ref{prePoisson}) is a
Poisson bracket. Such a bivector field is called a Poisson tensor.

Let $\Pi$ be a bivector field on a manifold $M$, in a local system
of coordinates $(x_1,\ldots, x_n)$ we have
\[
\Pi=\sum_{i<j}\Pi_{ij}\frac{\partial}{\partial
x_i}\wedge\frac{\partial}{\partial x_j}
=\frac{1}{2}\sum_{i,j}\Pi_{ij}\frac{\partial}{\partial x_i}
\wedge\frac{\partial}{\partial x_j},
\]
where $\Pi_{ij}=\Pi(dx_i,dx_j)$. The bivector field $\Pi$ is a
Poisson tensor if and only if it satisfies the following system of
equations :
$$
\oint_{ijk}\sum_{l}\frac{\partial\Pi_{ij}}{\partial x_l}\Pi_{lk}=0,
\quad\textrm{ for all } i,j,k,
$$
where $\oint_{ijk}A_{ijk}$ means the cyclic sum $A_{ijk}+A_{kij}+A_{jki}$.

Given a Poisson tensor $\Pi$ (or more generally, a bivector field)
on a manifold M, we can associate to it a natural homomorphism
called the anchor map :
$$
\sharp_{\Pi} : T^* M\rightarrow TM
$$
defined by
$$
\beta(\sharp_{\Pi}(\alpha))=\Pi{(\alpha,\beta)}
$$
for any $\alpha, \beta\in T^*M$. If $\varphi$ is a function, then
$\sharp_{\Pi}(d\varphi)$ is the Hamiltonian vector field of
$\varphi$.

Let $x$ be a point of $M$. The restriction of $\sharp_{\Pi}$ to the
 cotangent space $T_x^*M$ will be denoted by $\sharp_{\Pi(x)}$.
 The image $C_x =\textrm{Im}\, \sharp_{\Pi(x)}$ of
$\sharp_{\Pi(x)}$ is called the characteristic space of $\Pi$ at
$x$. The dimension $\textrm{rank}_x\Pi = \dim C_x$ of $C_x$ is
called the rank of $\Pi$ at $x$ and $\textrm{rank}\,\Pi =\max_{x\in
M}\{\textrm{rank}_x\Pi\}$ is called the rank of $\Pi$. When
$\textrm{rank}_x\Pi=\dim M$ we say that $\Pi$ is nondegenerate at
$x$. If $\textrm{rank}_x\Pi$ is a constant on $M$, i.e. does not
depend on $x$, then $\Pi$ is called regular.

If $\Pi$ is nondegenerate at all $x$ of $M$ then the anchor map
$\sharp_\Pi$ is invertible and defines a $2$-form $\omega$ on $M$ by
$\omega(X,Y)=\Pi(\sharp^{-1}(X),\sharp^{-1}(Y))$. The condition
$[\Pi,\Pi]_S=0$ is equivalent to $d\omega=0$, see \cite{Jean-Paul
Dufour}, and in this case, we say that $(M,\omega)$ is a symplectic
manifold.

Let $\Pi$ be a Poisson tensor on  a manifold $M$. The distribution
$\textrm{Im}\,\sharp_{\Pi}$ is integrable and defines a singular
foliation $\mathcal{F}$. The leaves of $\mathcal{F}$ are symplectic
immersed submanifolds of $M$. The foliation $\mathcal{F}$ is called
the symplectic foliation associated to the Poisson structure
$(M,\Pi)$.

\subsection{Contravariant connections and generalized Poisson brackets}

Contravariant connections associated to a Poisson structure have
recently turned out to be useful in several areas of Poisson
geometry. Contravariant connections were defined by Vaisman
\cite{Vais} and were analyzed in details by Fernandes \cite{fern}.
This notion appears extensively in the context of noncommutative
deformations (see \cite{Hawkins2,Hawkins}).

Let $M$ be a manifold. A bivector field $\Pi$ on $M$ induces on the
space of differential 1-forms $\Gamma(T^*M)$ a bracket $[.,.]_\Pi$
called the Koszul bracket :
$$
[\alpha,\beta]_{\Pi}=\mathcal{L}_{\sharp_{\Pi}(\alpha)}\beta-\mathcal{L}_{\sharp_{\Pi}(\beta)}
\alpha-d(\Pi(\alpha,\beta)).
$$
If $\Pi$ is a Poisson tensor, the bracket $[.,.]_\Pi$ is a Lie
bracket. For this Lie bracket on $\Gamma(T^*M)$ and for the usual
Lie bracket $[.,.]$ on the space of vector fields $\Gamma(TM)$, the
anchor map $\sharp_{\Pi}$ induces a Lie algebra homomorphism
$\sharp_{\Pi}: \Gamma(T^*M)\rightarrow \Gamma(TM)$, i.e.
\begin{equation}\label{Lie algebra homom}
\sharp_{\Pi}([\alpha,\beta]_{\Pi})=[\sharp_{\Pi}(\alpha),\sharp_{\Pi}(\beta))].
\end{equation}

A contravariant connection on $M$ with respect to the bivector field
$\Pi$ is an $\mathbb{R}$-bilinear map
$$
\begin{array}{rccc}
\mathcal{D}: & \Gamma(T^*M)\times\Gamma(T^*M) & \longrightarrow & \Gamma(T^*M), \\
& (\alpha,\beta) & \mapsto & \mathcal{D}_\alpha\beta%
\end{array}
$$
such that the mapping $\alpha \mapsto \mathcal{D}_{\alpha}\beta$ is
$C^\infty(M)$-linear, that is
$$
\mathcal{D}_{\varphi\alpha}\beta=\varphi \mathcal{D}_{\alpha}\beta
\quad \textrm{ for all }\varphi\in C^\infty(M),
$$
and that the mapping $\beta \mapsto \mathcal{D}_{\alpha}\beta$ is a
derivation in the following sense :
$$
\mathcal{D}_{\alpha}\varphi\beta=\varphi
\mathcal{D}_{\alpha}\beta+\sharp_{\Pi}(\alpha)(\varphi)\beta,
~\text{for} ~\text{all}~ \varphi\in C^\infty(M).
$$

As in the covariant case, for a differential $1$-form $\alpha$,  we
can use the derivation $\mathcal{D}_{\alpha}$ to define a
contravariant derivative of multivector fields of degree $r$ by
\begin{equation}\label{connection multi}
 (\mathcal{D}_{\alpha}Q)(\beta_1,\ldots,\beta_r)=\sharp_{\Pi}(\alpha)(Q(\beta_1,\ldots,\beta_r))-
 \sum_{i=1}^rQ(\beta_1,\ldots,\mathcal{D}_{\alpha}\beta_i,\ldots,\beta_r).
\end{equation}
In particular, if $X$ is a vector field, we have
\begin{equation}\label{connection champ}
(\mathcal{D}_{\alpha}X)(\beta)=\sharp_{\Pi}(\alpha)\beta(X)-\left(\mathcal{D}_{\alpha}\beta\right)(X).
\end{equation}
We can also define a contravariant derivative of $r$-forms by
\begin{equation}\label{connection forms}
(\mathcal{D}_{\alpha}\omega)(X_1,\ldots,X_r)=\sharp_{\Pi}(\alpha)(\omega(X_1,\ldots,X_r))-
 \sum_{i=1}^r\omega(X_1,\ldots,\mathcal{D}_{\alpha}X_i,\ldots,X_r).
\end{equation}

The definitions of the torsion and of the curvature of a
contravariant connection $\mathcal{D}$ are formally identical to the
 definitions in the covariant case (see \cite{fern}) :
$$
\mathcal{T}(\alpha,\beta)=\mathcal{D}_{\alpha}\beta-\mathcal{D}_{\beta}\alpha-[\alpha,\beta]_{\Pi}
\quad\textrm{ and }\quad
\mathcal{R}(\alpha,\beta)=\mathcal{D}_{\alpha}\mathcal{D}_{\beta}-
\mathcal{D}_{\beta}\mathcal{D}_{\alpha}-\mathcal{D}_{[\alpha,\beta]_{\Pi}}.
$$
We say that a connection $\mathcal{D}$ is torsion-free (resp. flat)
if its torsion $\mathcal{T}$ (resp. its curvature $\mathcal{R}$)
vanishes identically. We say that $\mathcal{D}$ is locally symmetric
if $\mathcal{D}\mathcal{R}=0$, i.e., for any
$\alpha,\beta,\gamma,\delta \in \Gamma(T^*M)$, we have :
\begin{equation}\label{locally-sym}
\left(\mathcal{D}_\alpha\mathcal{R}\right)(\beta,\gamma)\delta
:=\mathcal{D}_\alpha(\mathcal{R}(\beta,\gamma)\delta)
-\mathcal{R}(\mathcal{D}_\alpha\beta,\gamma)\delta
-\mathcal{R}(\beta,\mathcal{D}_\alpha\gamma)\delta -
\mathcal{R}(\beta,\gamma)\mathcal{D}_\alpha\delta=0.
\end{equation}

Let us briefly recall the definition of a generalized Poisson
bracket, introduced by E. Hawkins in \cite{Hawkins}. To a
contravariant connection $\mathcal{D}$, with respect to the bivector
field $\Pi$, Hawkins associated an $\mathbb{R}$-bilinear bracket on
the differential graded algebra of differential forms $\Omega^*(M)$
that generalizes the initial pre-Poisson bracket $\{.,.\}$ given by
(\ref{prePoisson}) on $C^\infty(M)$. This bracket, that we will then
call a generalized pre-Poisson bracket associated with the
contravariant connection $\mathcal{D}$ and that we will still denote
by $\{.,.\}$, is defined as follows. The bracket of a function and a
differential form is given by
\begin{equation}\label{bracketconnection}
\{\varphi,\eta\}=\mathcal{D}_{d\varphi}\eta,
\end{equation}
and it is extended to two any differential forms in a way it
satisfies the following properties :
\begin{enumerate}
\item[$\bullet$] $\{.,.\}$ is antisymmetric, i.e.
\begin{equation}\label{antisym}
\{\omega,\eta\}=-(-1)^{\deg\omega\cdot \deg\eta} \{\eta,\omega\},
\end{equation}
\item[$\bullet$] the differential $d$ is a derivation with respect to
$\{.,.\}$, i.e.
\begin{equation}\label{differentialbracket}
d\{\omega,\eta\}=\{d\omega,\eta\}+(-1)^{\deg\omega}\{\omega,d\eta\},
\end{equation}
\item[$\bullet$] $\{.,.\}$ satisfies the product rule
\begin{equation}\label{productrule}
\{\omega,\eta\wedge\lambda\}=\{\omega,\eta\}\wedge\lambda+(-1)^{\deg
\omega\cdot \deg\eta }\eta\wedge\{\omega,\lambda\}.
\end{equation}
\end{enumerate}
When $\mathcal{D}$ is flat, Hawkins showed that there is a $(2,
3)$-tensor $\mathcal{M}$ symmetric in the covariant arguments,
antisymmetric in the contravariant arguments and such that the
following two assertions are equivalent :
\begin{enumerate}
\item[\textit{(i)}] the generalized pre-Poisson bracket $\{.,.\}$
satisfies the graded Jacobi identity
\begin{equation}\label{gen bracket jacobi}
\{\{\omega,\eta\},\lambda\}-\{\omega,\{\eta,\lambda\}\}+(-1)^{\deg\omega.\deg
\eta}\{\eta,\{\omega,\lambda\}\}=0
\end{equation}
\end{enumerate}
and
\begin{enumerate}
\item[\textit{(ii)}] the tensor $\mathcal{M}$ vanishes identically.
\end{enumerate}

The tensor $\mathcal{M}$ is called the metacurvature of
$\mathcal{D}$ and it is given by
$$
\mathcal{M}(d\varphi,\alpha,\beta)=\{\varphi,\{\alpha,\beta\}\}
-\{\{\varphi,\alpha\},\beta\}-\{\{\varphi,\beta\},\alpha\}.
$$
If $\mathcal{M}$ vanishes identically, the contravariant connection
$\mathcal{D}$ is called metaflat and the bracket $\{.,.\}$ is called
the generalized Poisson bracket associated with $\mathcal{D}$.

\subsection{Pseudo-Riemannian Poisson manifolds}

Now, let $(M,\tilde{g})$ be a pseudo-Riemannian manifold. The metric
$\tilde{g}$ defines the musical isomorphisms
$$
\begin{array}{cccc}
\flat_{\tilde{g}}: & TM&  \rightarrow  & T^*M \\
& X & \mapsto & \tilde{g}(X,.)
\end{array}
$$
and its inverse $\sharp_{\tilde{g}}$. We define the cometric $g$ of
the metric $\tilde{g}$ by :
$$
g(\alpha,\beta)=\tilde{g}(\sharp_{\tilde{g}}(\alpha),\sharp_{\tilde{g}}(\beta)).
$$

For a bivector field  $\Pi$ on $M$, there exists a unique
contravariant connection $\mathcal{D}$ associated to the pair
$(g,\Pi)$ such that the metric $g$ is parallel with respect to
$\mathcal{D}$, i.e.
\begin{equation}\label{parallel}
\sharp_{\Pi}(\alpha).{g}(\beta,\gamma)={g}(\mathcal{D}_\alpha\beta,\gamma)+
g(\beta,\mathcal{D}_\alpha \gamma),
\end{equation}
and that $\mathcal{D}$ is torsion-free, i.e.
\begin{equation}\label{connection torsion free}
\mathcal{D}_\alpha\beta-\mathcal{D}_\beta\alpha=[\alpha,\beta]_{\Pi}.
\end{equation}
The connection $\mathcal{D}$ is called the Levi-Civita contravariant
connection associated with $(g,\Pi)$. It is characterized by the
Koszul formula :
\begin{equation}  \label{koszul}
\begin{split}
2g(\mathcal{D}_\alpha\beta,\gamma)=\sharp_{\Pi}(\alpha).g(\beta,\gamma)
+\sharp_{\Pi}(\beta).g(\alpha,\gamma)-\sharp_{\Pi}(\gamma).g(\alpha,\beta) \\
+g([\alpha,\beta]_{\Pi},\gamma)+{g}([\gamma,\alpha]_{\Pi},\beta)+g
([\gamma,\beta]_{\Pi}, \alpha).
\end{split}
\end{equation}
If $\mathcal{R}$ is the curvature of $\mathcal{D}$ and if $\theta_p$
and $\eta_p$ are two non-parallel cotangent vectors at $p\in M$ then
the number
$$
\mathcal{K}_p(\theta_p,\eta_p)=\dfrac{g_p(\mathcal{R}_p(\theta_p,\eta_p)\eta_p,
\theta_p)}{g_p(\theta_p,\theta_p)g_p(\eta_p,\eta_p)-g_p(\theta_p,\eta_p)^2}
$$
is called the sectional contravariant curvature of $(M,g,\Pi)$ at
$p$ in the direction of the plane spanned by the covectors
$\theta_p$ and $\eta_p$ in $T_p^*M$. Let $\{e_1,\ldots,e_n\}$ be a
local orthonormal basis of $T^*_pM$ with respect to $g$ on an open
$U\subset M$. Let $\theta_p$ and $\eta_p$ be two cotangent vectors
at $p\in M$. The Ricci curvature $r_p$ at $p$ and the scalar
curvature $\mathcal{S}_p$ of $(M,g,\Pi)$ at $p$ are defined by
$$
r_p(\theta_p,\eta_p)=\sum_{i=1}^n
g_p(\mathcal{R}_p(\theta_p,e_i)e_i,\eta_p) \textrm{ and }
\mathcal{S}_p=\sum_{j=1}^n \sum_{i=1}^n
g_p(\mathcal{R}_p(e_i,e_j)e_j,e_i).
$$

With the notations above, we say that the triple $(M,g,\Pi)$ is a
pseudo-Riemannian Poisson manifold if $\mathcal{D}\Pi=0$, i.e., for
any $\alpha,\beta,\gamma\in \Gamma(T^*M)$, we have
$$
\sharp_{\Pi}(\alpha).\Pi(\beta,\gamma)-\Pi(\mathcal{D}_\alpha\beta,\gamma)-\Pi(\beta,\mathcal{D}_\alpha
\gamma)=0.
$$
The notions of Levi-Civita contravariant connection and of
pseudo-Riemannian Poisson manifold were introduced by Boucetta in
\cite{Bou1,Bou2}.

We say that the triple $(M,g,\Pi)$ is flat (resp. locally symmetric)
if  $\mathcal{R}=0$ (resp. $\mathcal{D}\mathcal{R}=0$). We say that
the triple $(M,g,\Pi)$ is metaflat if $\mathcal{R}=0$ and
$\mathcal{M}=0$, where $\mathcal{M}$ is the metacurvature of the
Levi-Civita contravariant connection $\mathcal{D}$.

In this paper, for a pseudo-Riemannian manifold $(M,\tilde{g})$ and
a bivector field $\Pi$ on $M$, we will always denote by
$J\in\Gamma(M, TM\otimes T^*M)$ the field of homomorphisms defined
by
\begin{equation}\label{gJPi}
g(J\alpha,\beta)=\Pi(\alpha,\beta).
\end{equation}
If $(M,g,\Pi)$ is a pseudo-Riemannian Poisson manifold and
$\mathcal{D}$ the Levi-Civita contravariant connection associated
with $(g,\Pi)$ then $\mathcal{D}J=0$, i.e., for any $\alpha,\beta
\in \Gamma(T^*M)$,
$$
\mathcal{D}_\alpha(J\beta)=J\mathcal{D}_\alpha \beta.
$$

\section{Some differential operators associated to the pair $(g,\protect\Pi)$}

Let $(M,\tilde{g})$ be a pseudo-Riemannian manifold, $g$ be the
cometric of $\tilde{g}$ and $\Pi$ be a bivector field on $M$. In
this section, we define some new differential operators on $M$.

\begin{definition}
Let $\mathcal{D}$ be the contravariant Levi-Civita connection
associated with the pair $(g,\Pi)$. We define the contravariant
Hessian $H^\varphi_{\Pi}$ of a function $\varphi\in C^{\infty}(M)$
with respect to the bivector field $\Pi$ by
$$
H^\varphi_{\Pi}=\mathcal{D}\mathcal{D}\varphi,
$$
i.e., the contravariant Hessian $H^\varphi_{\Pi}$ of $\varphi$ is
its second contravariant differential.
\end{definition}

\begin{proposition}
The contravariant Hessian $H_{\Pi }^{\varphi }$ of $\varphi $ is a
$(0,2)$-tensor field and we have
$$
H_{\Pi }^{\varphi }(\alpha ,\beta )=\sharp _{\Pi }(\alpha )(\sharp
_{\Pi }(\beta )(\varphi ))-\sharp _{\Pi }(\mathcal{D}_{\alpha }\beta
)(\varphi )=-g(\mathcal{D}_{\alpha }Jd\varphi,\beta ).
$$
Moreover, when $\Pi$ is Poisson, the Hessian $H^\varphi_{\Pi}$ is
symmetric.
\end{proposition}

\begin{proof}
We have
$H^\varphi_{\Pi}(\alpha,\beta)=\mathcal{D}\mathcal{D}\varphi(\alpha,\beta)=
\left(\mathcal{D}_\alpha(\mathcal{D}\varphi)\right)(\beta)$. From
(\ref{connection champ}) and since we have
$\mathcal{D}\varphi=d\varphi\circ\sharp_{\Pi}$, i.e.
 $\mathcal{D}_{\alpha}\varphi=\sharp_{\Pi}(\alpha)(\varphi)$, we
 deduce that
\begin{equation}\label{f1}
H_\Pi^\varphi(\alpha,\beta)=\sharp_{\Pi}(\alpha)\left(\mathcal{D}_\beta\varphi\right)
-\mathcal{D}_{\mathcal{D}\alpha\beta}\varphi=
\sharp_\Pi(\alpha)\left(\sharp_\Pi(\beta)(\varphi)\right)-\sharp_\Pi(\mathcal{D}_\alpha\beta)(\varphi).
\end{equation}
To prove the second equality, notice that, for any $1$-form $\gamma$
on $M$, we have
$$
\sharp_\Pi(\gamma)(\varphi)=\Pi(\gamma,d\varphi)=-\Pi(d\varphi,\gamma)=-g(Jd\varphi,\gamma)=-g(\gamma,Jd\varphi),
$$
therefore, substituting in (\ref{f1}), taking $\gamma=\beta$ and
$\gamma=\mathcal{D}_\alpha\beta$, we get
$$
H^\varphi_{\Pi}(\alpha,\beta)=-\sharp_{\Pi}(\alpha)(g(\beta,Jd\varphi))
 +g(\mathcal{D}_{\alpha}\beta,Jd\varphi),
$$
and using (\ref{parallel}), we get
$H^\varphi_{\Pi}(\alpha,\beta)=-g(\mathcal{D}_{\alpha}Jd\varphi,\beta)$.

Now, assume that $\Pi$ is Poisson. By (\ref{Lie algebra homom}) and
since $\mathcal{D}$ is torsion-free (\ref{connection torsion free}),
we have
{\small$$
\sharp_\Pi(\mathcal{D}_\alpha\beta)=\sharp_\Pi\left(\mathcal{D}_\beta\alpha+\left[\alpha,\beta\right]_\Pi\right)
=\sharp_\Pi\left(\mathcal{D}_\beta\alpha\right)+\sharp_\Pi\left(\left[\alpha,\beta\right]_\Pi\right)
=\sharp_\Pi\left(\mathcal{D}_\beta\alpha\right)+\left[\sharp_\Pi(\alpha),\sharp_\Pi(\beta)\right],
$$}and by substituting in (\ref{f1}), we get
$$
\begin{array}{ll}
H_\Pi^\varphi(\alpha,\beta)
&=\sharp_\Pi(\alpha)\left(\sharp_\Pi(\beta)(\varphi)\right)
-\sharp_\Pi\left(\mathcal{D}_\beta\alpha\right)(\varphi)-\left[\sharp_\Pi(\alpha),\sharp_\Pi(\beta)\right](\varphi)\\
&=\sharp_\Pi(\beta)\left(\sharp_\Pi(\alpha)(\varphi)\right)
-\sharp_\Pi\left(\mathcal{D}_\beta\alpha\right)(\varphi),
\end{array}
$$
therefore, $H_\Pi^\varphi$ is symmetric.
\end{proof}

Observe that in case $(M,g,\Pi)$ is a pseudo-Riemannian Poisson
manifold, we have
$$
H_\Pi^\varphi(\alpha,\beta)=-g(J\mathcal{D}_{\alpha}d\varphi,\beta)
=-\Pi(\mathcal{D}_\alpha
d\varphi,\beta)=\Pi(\beta,\mathcal{D}_\alpha d\varphi).
$$

\begin{proposition}
If for any $\varphi \in C^{\infty }(M)$ we put
$$
\triangleleft_{\Pi}(\varphi)=\textrm{tr}_{g}(\alpha \mapsto
\mathcal{D}_{\alpha }Jd\varphi ) \quad\textrm{ and }\quad
\triangleright_{\Pi}(\varphi)=\textrm{tr}_{g}(\alpha \mapsto
\mathcal{D}_{\alpha}d\varphi),
$$
where $\textrm{tr}_{g}$ denotes the trace with respect to $g$, then
$\triangleleft _{\Pi }$ is a differential operator of degree two on
$C^{\infty}(M)$ and $\triangleright_{\Pi }$ is a vector field on
$M$.
\end{proposition}
\begin{proof}
One can easily see that both $\triangleleft _{\Pi}$ and
$\triangleright_{\Pi}$ are $\mathbb{R}$-linear. Now, let
$\{dx_{1},\ldots,dx_{n}\}$ be a local, orthonormal basis of
$1$-forms and let us show that $\triangleleft_\Pi$ is a differential
operator of degree two on $C^{\infty}(M)$. If $\varphi\in
C^{\infty}(M)$, we have
$$
\triangleleft_{\Pi}(\varphi)=\sum_{i=1}^{n}g(\mathcal{D}_{dx_{i}}Jd\varphi
,dx_{i})=\sum_{i,j=1}g(\mathcal{D}_{dx_i}(\frac{\partial\varphi}{\partial
x_j}Jdx_j),dx_i).
$$
Since by (\ref{gJPi}) we have $Jdx_j=\sum_{k=1}^n\Pi_{jk}dx_k$, then
$$
\mathcal{D}_{dx_i}(\frac{\partial\varphi}{\partial
x_j}Jdx_j)=\sum_{k=1}^n\left(\sharp_{\Pi}(dx_i)(\frac{\partial\varphi}{\partial
x_j}\Pi_{jk})dx_k +\frac{\partial\varphi}{\partial
x_j}\Pi_{jk}\mathcal{D}_{dx_i}dx_k\right).
$$
Therefore,
$$
\triangleleft _{\Pi
}(\varphi)=\sum_{i,j,k=1}^n\Pi_{ik}\frac{\partial}{\partial
x_k}(\frac{\partial\varphi}{\partial x_j}\Pi_{ji})
+\sum_{i,j,k=1}^n\frac{\partial\varphi}{\partial
x_j}\Pi_{jk}\Gamma_{ik}^i.
$$
Finally, as $\Gamma_{ik}^i=0$, we get
$$
\triangleleft _{\Pi }(\varphi)=\sum_{i,j,k=1}^n\frac{\partial
\Pi_{ij}}{\partial x_k}\Pi_{ki}\frac{\partial \varphi}{\partial
x_j}+\sum_{i,j,k=1}^n\Pi_{ij}\Pi_{ki}\frac{\partial^2\varphi}{\partial
x_j\partial x_k}.
$$
To show that $\triangleright_{\Pi } $ corresponds to a vector field
one has to verify that it satisfies the Leibnitz rule with respect
to $\varphi$, i.e. that for any $\varphi, \psi\in
\mathcal{C}^{\infty}(M)$, we have
$$
\triangleright _{\Pi }(\varphi\psi
)=\varphi\sum_{i=1}^{n}g(\mathcal{D}_{dx_{i}}d\psi ,dx_{i})
+\psi\sum_{i=1}^{n}g(\mathcal{D}_{dx_{i}}d\varphi  ,dx_{i}).
$$
A direct calculation using the definition of $\mathcal{D}$ gives
$$
\triangleright _{\Pi }(\varphi\psi )-\varphi\triangleright _{\Pi
}(\psi )-\psi\triangleright _{\Pi }(\varphi )
=\sum_{i,j=1}^n\Pi_{ij}(\frac{\partial\psi}{dx_i}\frac{\partial\varphi}{dx_j}
+\frac{\partial\psi}{dx_j}\frac{\partial\varphi}{dx_i}).
$$
It is clear that the right hand side of this formula is zero.
\end{proof}

\section{Poisson brackets on a product manifold}
\subsection{Horizontal and vertical lifts}

Throughout this paper $M_{1}$ and $M_{2}$ will be respectively
$m_{1}$ and $m_{2}$ dimensional manifolds, $M_1\times M_2$ the
product manifold with the natural product coordinate system and
$\sigma _{1}:M_{1}\times M_{2}\rightarrow M_{1}$ and $\sigma
_{2}:M_{1}\times M_{2}\rightarrow M_{2}$ the usual projection maps.

We recall briefly how the calculus on the product manifold $M_1
\times M_2$ derives from that of $M_1$ and $M_2$ separately. For
details see \cite{ONeil}.

Let $\varphi _{1}$ in $C^{\infty }(M_{1})$. The horizontal lift of
$\varphi_{1}$ to $M_{1}\times M_{2}$ is $\varphi_{1}^{h}=\varphi
_{1}\circ \sigma _{1}$. One can define the horizontal lifts of
tangent vectors as follows. Let $p\in M_1$ and let $X_{p}\in
T_{p}M_{1}$. For any $q\in M_{2}$ the horizontal lift
 of $X_{p}$ to $(p,q)$ is the unique tangent vector $X_{(p,q)}^{h}$
in $T_{(p,q)}(M_{1}\times M_2)$ such that
$d_{(p,q)}\sigma_{1}(X_{(p,q)}^{h})=X_{p}$ and
$d_{(p,q)}\sigma_{2}(X_{(p,q)}^{h})=0$. We can also define the
horizontal lifts of vector
 fields as follows. Let $X_1\in \Gamma (TM_{1})$. The horizontal lift of $X_1$
 to $M_{1}\times M_{2}$ is the vector field
$X_1^{h}\in \Gamma (T(M_{1}\times M_{2}))$ whose value at each
$(p,q)$ is the horizontal lift of the tangent vector $(X_1){p}$ to
$(p,q)$. For $(p,q)\in M_1\times M_2$, we will denote the set of the
horizontal lifts to $(p,q)$ of all the tangent vectors of $M_{1}$ at
$p$ by $L(p,q)(M_{1})$. We will denote the set of the horizontal
lifts of all vector fields on $M_{1}$ by $\mathfrak{L}(M_{1})$.

The vertical lift $\varphi_2^v$ of a function $\varphi_2\in
C^{\infty}(M_2)$ to $M_1\times M_2$ and the vertical lift $X_2^v$ of
a vector field $X_2\in \Gamma (TM_{2})$ to $M_1\times M_2$ are
defined in the same way using the projection $\sigma_2$. Note that
the spaces $\mathfrak{L}(M_{1})$ of the horizontal lifts and
$\mathfrak{L}(M_{2})$ of the vertical lifts are vector subspaces of
$\Gamma (T(M_{1}\times M_{2}))$ but neither is invariant under
multiplication by arbitrary functions $\varphi \in C^{\infty
}(M_{1}\times M_{2})$.

We define the horizontal lift of a covariant tensor $\omega_1$ on
$M_1$ to be its pullback $\omega_1^h$ to $M_1\times M_2$ by the
means of the projection map $\sigma_1$, i.e.
$\omega_1^h:=\sigma_1^*(\omega_1)$. In particular, for a $1$-form
$\alpha_1$ on $M_1$ and a vector field $X$ on $M_1\times M_2$, we
have
$$
(\alpha_1^h)(X)=\alpha_1(d\sigma_1(X)).
$$
Explicitly, if $u$ is a tangent vector to $M_1\times M_2$ at
$(p,q)$, then
$$
(\alpha_1^h)_{(p,q)}(u)=(\alpha_1)_p(d_{(p,q)}\sigma_1(u)).
$$
Similarly, we define the vertical lift of a covariant tensor $w_2$
on $M_2$ to be its pullback $\omega_2^v$ to $M_1\times M_2$ by the
means of the projection map $\sigma_2$.

Observe that if $\{dx_1,\ldots,dx_{n_1}\}$ is the local basis of the
$1$-forms relative to a chart $(U,\Phi)$ of $M_1$ and
$\{dy_1,\ldots,dy_{n_2}\}$ the local basis of the $1$-forms relative
to a chart $(V,\Psi)$ of $M_2$, then
$\{(dx_1)^h,\ldots,(dx_{n_1})^h,(dy_1)^v,\ldots,(dy_{n_2})^v\}$ is
the local basis of the $1$-forms relative to the chart $(U\times
V,\Phi\times \Psi)$ of $M_1\times M_2$.

Let $Q_1$ (resp. $Q_2$) be an $r$-contravariant tensor on $M_1$
(resp. on $M_2$). We define the horizontal lift $Q_1^h$ of $Q_1$
(resp. the vertical lift $Q_2^v$ of $Q_2$) to $M_1\times M_2$ by
$$
\left\{\begin{array}{ll} Q_1^h(\alpha_1^h,\ldots
,\alpha_r^h)=\left[Q_1(\alpha_1,\ldots ,\alpha_r)\right]^h \textrm{ and}  \\
i_{\beta^v}Q_1^h=0, \quad \forall \beta\in\Gamma(T^*M_2),
\end{array}\right.
$$
resp.
$$
\left\{\begin{array}{ll} Q_2^v(\beta_1^v,\ldots
,\beta_r^v)=\left[Q_2(\beta_1,\ldots ,\beta_r)\right]^v \textrm{ and}  \\
i_{\alpha^h}Q_2^v=0, \quad \forall \alpha\in\Gamma(T^*M_1),
\end{array}\right.
$$
where $i$ denotes the inner product. The following lemma will be
useful for our computations.

\begin{lemma} \label{lift} $\;$
\begin{enumerate}
\item Let $\varphi_i\in C^\infty(M_i)$, $X_i,Y_i\in \Gamma (TM_{i})$
and $\alpha _{i}\in \Gamma (T^* M_{i})$, $i=1,2$. Let
$\varphi=\varphi_1^h+\varphi_2^v$, $X=X_{1}^{h}+X_{2}^{v}$ and
$\alpha ,\beta \in \Gamma (T^*(M_{1}\times M_{2}))$. Then
\begin{enumerate}
\item[i/] For all $(i,l)\in \{(1,h),(2,v)\}$, we have
$$
X_i^l(\varphi)=X_i(\varphi_i)^l,\quad [X,Y_i^l]=[X_i,Y_i]^l \quad
\textrm{ and } \quad  \alpha _{i}^{l}(X)=\alpha_{i}(X_{i})^{l}.
$$
\item[ii/] If for all $(i,l)\in \{(1,h),(2,v)\}$ we have $\alpha
(X_{i}^{l})=\beta (X_{i}^{l})$, then $\alpha =\beta$.
\end{enumerate}
\item Let $\omega_i$ and $\eta_i$ be $r$-forms on $M_i$, $i=1,2$.
Let $\omega=\omega_1^h+\omega_2^v$ and $\eta=\eta_1^h+\eta_2^v$. We
have
$$
d\omega=(d\omega_1)^h+(d\omega_2)^v \quad \textrm{ and } \quad
\omega \wedge \eta=(\omega_1\wedge \eta_1)^h+(\omega_2\wedge
\eta_2)^v.
$$
\end{enumerate}
\end{lemma}
\begin{proof}
See \cite{Nas}.
\end{proof}

\subsection{The warped Poisson tensor}

Now, we construct a bivector field on a product manifold and give
the conditions under which it becomes a Poisson tensor.

Let $\Pi_1$ and $\Pi_2$ be bivector fields on $M_1$ and $M_2$
respectively. Given a smooth function $\mu$ on $M_1$, we define a
bivector field $\Pi^{^\mu}$ on $M_1\times M_2$ by
$\Pi^{^\mu}=\Pi_1^h+\mu^h\Pi_2^v$. It is the unique bivector field
such that
$$
\Pi^{^\mu}(\alpha_{1}^{h},\beta_{1}^{h})=\Pi_{1}(\alpha_{1},\beta_{1})^{h},
\;
\Pi^{^\mu}(\alpha_{2}^{v},\beta_{2}^{v})=\mu^h\Pi_{2}(\alpha_{2},\beta_{2})^{v}
\textrm{ and } \Pi^{^\mu}(\alpha_{1}^{h},\beta_{2}^{v})=0,
$$
for any $\alpha_i,\beta_i\in \Gamma(T^*M_i)$, $i=1,2$. We call
$\Pi^{^\mu}$ the warped bivector field relative to $\Pi_1,\Pi_2$ and
the warping function $\mu$.

\begin{proposition}\label{bracket}
Let $\alpha_i,\beta_i\in\Gamma(T^*M_i)$, $ i=1,2$. Let
$\alpha=\alpha_1^h+\alpha_2^v$ and $\beta=\beta_1^h+\beta_2^v$. Then
\begin{enumerate}
\item $\sharp_{\Pi^{^\mu}}(\alpha)=[\sharp_{\Pi_1}(\alpha_1)]^h+\mu^h[
\sharp_{\Pi_2}(\alpha_2)]^v$,
\item $\mathcal{L}_{\sharp_{\Pi^{^\mu}}(\alpha)}\beta=(\mathcal{L}
_{\sharp_{\Pi_1}(\alpha_1)}\beta_1)^h+
\mu^h(\mathcal{L}_{\sharp_{\Pi_2}(\alpha_2)}\beta_2)^v
+\Pi_2(\alpha_2,\beta_2)^v(d\mu)^h$,
\item $[\alpha,\beta]_{\Pi^{^\mu}}=[\alpha_1,\beta_1]_{\Pi_1}^h+\mu^h[
\alpha_2,\beta_2]_{\Pi_2}^v+\Pi_2(\alpha_2,\beta_2)^v(d\mu)^h$.
\end{enumerate}
\end{proposition}
\begin{proof}
1. By Lemma  \ref{lift}, for any $\gamma_i\in \Gamma(T^*M_i)$,
$i=1,2$, we have
$$
\left\{\begin{array}{l}
\gamma_1^h(\sharp_{\Pi^{^\mu}}(\alpha_1^h))=\Pi^{^\mu}(\alpha_1^h,\gamma_1^h)=\Pi_1(\alpha_1,\gamma_1)^h
=(\gamma_1(\sharp_{\Pi_1}(\alpha_1)))^h=\gamma_1^h([\sharp_{\Pi_1}(\alpha_1)]^h)\\
\gamma_2^v(\sharp_{\Pi^{^\mu}}(\alpha_1^h))=\Pi^{^\mu}(\alpha_1^h,\gamma_2^v)=0=\gamma_2^v([\sharp_{\Pi_1}(\alpha_1)]^h)
\end{array}\right.
$$
and similarly
$$
\left\{\begin{array}{l}
\gamma_1^h(\sharp_{\Pi^{^\mu}}(\alpha_2^v))=\Pi^{^\mu}(\alpha_2^v,\gamma_1^h)=0
=\gamma_1^h(\mu^h[\sharp_{\Pi_2}(\alpha_2)]^v)\\
\gamma_2^v(\sharp_{\Pi^{^\mu}}(\alpha_2^v))=\Pi^{^\mu}(\alpha_2^v,\gamma_2^v)=\mu^h\Pi_2(\alpha_2,\gamma_2)^v
=\gamma_2^v(\mu^h[\sharp_{\Pi_2}(\alpha_2)]^v).
\end{array}\right.
$$
Therefore
$\sharp_{\Pi^{^\mu}}(\alpha)=\sharp_{\Pi^{^\mu}}(\alpha_1^h)+
\sharp_{\Pi^{^\mu}}(\alpha_2^v)=[\sharp_{\Pi_1}(\alpha_1)]^h+\mu^h[\sharp_{\Pi_2}(\alpha_2)]^v$. \smallskip \\
2. Using the assertion 1., the Leibniz identity and Lemma
\ref{lift}, we have
\begin{align*}
\mathcal{L}_{\sharp_{\Pi^{^\mu}}(\alpha)}\beta
&=\mathcal{L}_{[\sharp_{\Pi_1}(\alpha_1)]^h}\beta+\mathcal{L}_{\mu^h[\sharp_{\Pi_2}(\alpha_2)]^v}\beta
\\
&=(\mathcal{L}_{\sharp_{\Pi_1}(\alpha_1)}\beta_1)^h +\mu^h
\mathcal{L}_{[\sharp_{\Pi_2}(\alpha_2)]^v}\beta+\beta([\sharp_{\Pi_2}(\alpha_2)]^v)d\mu^h\\
&=(\mathcal{L}_{\sharp_{\Pi_1}(\alpha_1)}\beta_1)^h+
\mu^h(\mathcal{L}_{\sharp_{\Pi_2}(\alpha_2)}\beta_2)^v
+\Pi_2(\alpha_2,\beta_2)^v(d\mu)^h.
\end{align*}
3. It is a direct consequence of 2.
\end{proof}

The following result provide a necessary and sufficient condition
for the  bivector field $\Pi^{^\mu}$ to be a Poisson tensor.

\begin{theorem}\label{tensor}
Let $(M_1,\Pi_1)$ and $(M_2,\Pi_2)$ be two Poisson manifolds such
that $\Pi_2$ is non trivial and let $\mu$ be a smooth function on
$M_1$. Then $(M_1\times M_2, \Pi^{^\mu})$ is a Poisson manifold if
and only if $\mu$ is a Casimir function.
\end{theorem}

\begin{proof}
A straightforward calculation using Lemma \ref{lift} shows that
\begin{align*}
[\Pi^{^\mu},\Pi^{^\mu}]_S((d\varphi_1)^h,(d\phi_1)^h,(d\psi_1)^h)&
=([\Pi_1,\Pi_1]_S((d\varphi_1),(d\phi_1),(d\psi_1)))^h\\
[\Pi^{^\mu},\Pi^{^\mu}]_S((d\varphi_2)^v,(d\phi_2)^v,(d\psi_2)^v)&
=(\mu^2)^h([\Pi_2,\Pi_2]_S((d\varphi_2),(d\phi_2),d\psi_2))^v\\
[\Pi^{^\mu},\Pi^{^\mu}]_S((d\varphi_1)^h,(d\phi_1)^h,(d\psi_2)^v)&=0\\
[\Pi^{^\mu},\Pi^{^\mu}]_S((d\varphi_1)^h,(d\phi_2)^v,(d\psi_2)^v)&=
[X_{\mu}(\varphi_1)]^h \Pi_2(d\phi_2,d\psi_2)^v
\end{align*}
for any $\varphi_i,\phi_i,\psi_i\in C^{\infty}(M_i)$, $i=1,2$. Since
$\Pi_1$ and $\Pi_2$ are Poisson tensors, then
$$
[\Pi^{^\mu},\Pi^{^\mu}]_S((d\varphi_1)^h,(d\phi_1)^h,(d\psi_1)^h)=
[\Pi^{^\mu},\Pi^{^\mu}]_S((d\varphi_2)^v,(d\phi_2)^v,(d\psi_2)^v)=0.
$$
Therefore, $\Pi^{^\mu}$ is a Poisson tensor if and only if
$X_\mu=0$.
\end{proof}

\begin{definition}
Let $(M_1,\Pi_1)$ and $(M_2,\Pi_2)$ be two Poisson manifolds such
that $\Pi_2\neq 0$ and let $\mu$ be a Casimir function on
$(M_1,\Pi_1)$. The Poisson tensor $\Pi^{^\mu}$ on $M_1\times M_2$ is
called the warped Poisson tensor relative to $\Pi_1$, $\Pi_2$ and
the warping function $\mu$.
\end{definition}

\begin{corollary}\textbf{(The symplectic case)}
Under the assumptions of Theorem \ref{tensor}. If $\Pi_1$ and
$\Pi_2$ are nondegenerate and $\mu$ is nonvanishing, then
$(M_1\times M_2, \Pi^{^\mu})$ is symplectic if and only if $\mu$ is
essentially constant (i.e. constant on each connected component).
\end{corollary}
\begin{proof}
From Theorem \ref{tensor}, $(M_1\times M_2, \Pi^{^\mu})$ is a
Poisson manifold  if and only if $\mu$ is a Casimir function. Since
$\Pi_1$ is nondegenerate, the only Casimir functions on
$(M_1,\Pi_1)$ are the essentially constant functions. Since $\mu$ is
nonvanishing we have $\textrm{rank}\,\mu\Pi_2=\textrm{rank}\,\Pi_2$
and then
$$
\textrm{rank}\,\Pi^{^\mu}=
\textrm{rank}\,\Pi_1+\textrm{rank}\,\mu\Pi_2=\dim M_1 + \dim M_2
=\dim (M_1\times M_2).
$$
Therefore $(M_1\times M_2,\Pi^{^\mu})$ is a Poisson manifold if and
only if it is symplectic.
\end{proof}

\subsection{The warped generalized Poisson bracket}\label{wgPb}

Let $\Pi_1$ and $\Pi_2$ be bivector fields on $M_1$ and $M_2$
respectively. Let $\mathcal{D}^1$ and $\mathcal{D}^2$ be
contravariant connections on $(M_1,\Pi_1)$ and $(M_2,\Pi_2)$
respectively. Let $\mathcal{D}^\mu$
 be the contravariant connection on $M_1\times M_2$ with respect to $\Pi^\mu$
 given by
$$
\mathcal{D}^\mu_{\alpha_1^h}\beta_1^h=(\mathcal{D}^1_{\alpha_1}\beta_1)^h,
\qquad
\mathcal{D}^\mu_{\alpha_2^v}\beta_2^v=\mu^h(\mathcal{D}^2_{\alpha_2}\beta_2)^v
\quad\textrm{ and }\quad  \mathcal{D}^\mu_{\alpha_1^h}\beta_2^v=
\mathcal{D}^\mu_{\alpha_2^v}\beta_1^h=0,
$$
for any $\alpha_i,\beta_i\in \Gamma(T^*M_i)$, $i=1,2$.

\begin{proposition}
\begin{enumerate}
\item Let $\mathcal{T}_i$ be the torsion of $\mathcal{D}^i$,
$i=1,2$, and let $\mathcal{T}_\mu$ be the torsion of
$\mathcal{D}^\mu$. We have
$$
\mathcal{T}_\mu = \mathcal{T}_1^h + \mu^h \mathcal{T}_2^v - (d\mu)^h
\Pi_2^v.
$$
\item Let $\mathcal{R}_i$ be the curvature of $\mathcal{D}^i$,
$i=1,2$ and let $\mathcal{R}_\mu$ be the curvature of
$\mathcal{D}^\mu$. Let $\alpha_i,\beta_i,\gamma_i\in
\Gamma(T^*M_i)$, $i=1,2$, and let $\alpha=\alpha_1^h+\alpha_2^v$,
$\beta=\beta_1^h+\beta_2^v$ and $\gamma=\gamma_1^h+\gamma_2^v$. We
have
$$
\mathcal{R}_\mu(\alpha,\beta)\gamma=
\left[\mathcal{R}_1(\alpha_1,\beta_1)\gamma_1\right]^h + (\mu^2)^h
\left[\mathcal{R}_2(\alpha_2,\beta_2)\gamma_2\right]^v.
$$
\end{enumerate}
\end{proposition}
\begin{proof}
Use the definition of $\mathcal{D}^\mu$ and 3. of Proposition
\ref{bracket}.
\end{proof}

Therefore, if $\mu$ is nonzero, $\mathcal{D}^\mu$ is flat if and
only if $\mathcal{D}^1$ and $\mathcal{D}^2$ are flat and, if
$d\mu=0$, $\mathcal{D}^\mu$ is torsion-free if and only if
$\mathcal{D}^1$ and $\mathcal{D}^2$ are.

\begin{proposition}\label{generalized bracket}
Let $\{.,.\}_1$, $\{.,.\}_2$ and $\{.,.\}_\mu$ be the generalized
pre-Poisson brackets associated with $\mathcal{D}^1$,
$\mathcal{D}^2$ and $\mathcal{D}^\mu$ respectively. We have :
$$
\{\omega_1^h,\eta_1^h\}_\mu= \{\omega_1,\eta_1\}_1^h, \qquad
\{\omega_2^v,\eta_2^v\}_\mu= \mu^h\{\omega_2,\eta_2\}_2^v
\quad\textrm{ and }\quad \{\omega_1^h,\eta_2^v\}_\mu=0,
$$
for any differential forms $\omega_i,\eta_i\in \Omega^*(M_i)$,
$i=1,2$.
\end{proposition}
\begin{proof}
Let us first, once and for all, say that all along the computations
we use Lemma \ref{lift}. Now, by (\ref{connection champ}),
(\ref{connection forms}) and (\ref{bracketconnection}), we get for a
function and a differential form:
$$
\{\varphi_1^h,\eta_1^h\}_\mu=\{\varphi_1,\eta_1\}_1^h, \qquad
\{\varphi_2^v,\eta_2^v\}_\mu=\mu^h\{\varphi_2,\eta_2\}_2^v
$$
and
$$
\{\varphi_1^h,\eta_2^v\}_\mu=\{\varphi_2^v,\eta_1^h\}_\mu=0.
$$
By the Leibniz identity (\ref{differentialbracket}) and using the
identities above, we get for an exact $1$-form and a differential
form :
$$
\{(d\varphi_1)^h,\eta_1^h\}_\mu=\{d\varphi_1,\eta_1\}_1^h, \qquad
\{(d\varphi_2)^v,\eta_2^v\}_\mu=\mu^h\{d\varphi_2,\eta_2\}_2^v
$$
and
$$
\{(d\varphi_1)^h,\eta_2^v\}_\mu=\{(d\varphi_2)^v,\eta_1^h\}_\mu=0.
$$
Using the antisymmetry (\ref{antisym}), the product identity
(\ref{productrule}), and the identities above we get
$$
\{(\varphi_1d\psi_1)^h,\eta_1^h\}_\mu=\{\varphi_1d\psi_1,\eta_1\}_1^h,
\qquad
\{(\varphi_2d\psi_2)^v,\eta_2^v\}_\mu=\mu^h\{\varphi_2d\psi_2,\eta_2\}_2^v
$$
and
$$
\{(\varphi_1d\psi_1)^h,\eta_2^v\}_\mu=\{(\varphi_2d\psi_2)^v,\eta_1^h\}_\mu=0,
$$
thus for a $1$-form and any differential form we have
$$
\{\alpha_1^h,\eta_1^h\}_\mu=\{\alpha_1,\eta_1\}_1^h, \qquad
\{\alpha_2^v,\eta_2^v\}_\mu=\mu^h\{\alpha_2,\eta_2\}_2^v
$$
and
$$
\{\alpha_1^h,\eta_2^v\}_\mu=\{\alpha_2^v,\eta_1^h\}_\mu=0.
$$
Using again the antisymmetry (\ref{antisym}), the product identity
(\ref{productrule}) and the identities above we get for two
$1$-forms and any differential form :
$$
\{(\alpha_1\wedge
\beta_1)^h,\eta_1^h\}_\mu=\{\alpha_1\wedge\beta_1,\eta_1\}_1^h,
\qquad \{(\alpha_2\wedge
\beta_2)^v,\eta_2^v\}_\mu=\mu^h\{\alpha_2\wedge \beta_2,\eta_2\}_2^v
$$
and
$$
\{(\alpha_1\wedge\beta_1)^h,\eta_2^v\}_\mu=\{(\alpha_2\wedge\beta_2)^v,\eta_1^h\}_\mu=0.
$$
Now, by induction we get the identities of the proposition.
\end{proof}

\begin{corollary}\label{generalized Poisson}
If $\mu$ is nonzero, the bracket $\{.,.\}_\mu$ is a generalized
Poisson bracket if and only if the brackets $\{.,.\}_1$ and
$\{.,.\}_2$ are.
\end{corollary}
\begin{proof}
Indeed, by the proposition above we can see that the bracket
$\{.,.\}_\mu$ satisfy the graded Jacobi identity (\ref{gen bracket
jacobi}) if and only if the two brackets $\{.,.\}_1$ and $\{.,.\}_2$
do.
\end{proof}

\subsection{Other remarkable Poisson tensors on a product manifold}

\begin{proposition}
Let $\Pi_1$ and $\Pi_2$ be two bivector fields on $M_1$ and $M_2$
respectively. Let $f_1$ and $f_2$ be smooth functions on $M_1$ and
$M_2$ respectively and let $X_{f_1}=\sharp_{\Pi_1}(df_1)$ and
$X_{f_2}=\sharp_{\Pi_2}(df_2)$ be the corresponding Hamiltonian
fields. The bivector field $\Pi_{f_1,f_2}= X_{f_1}^h \wedge
X_{f_2}^v $ is a Poisson tensor on $M_1\times M_2$.
\end{proposition}
\begin{proof}
Using the properties of the Schouten-Nijenhuis bracket we get
$$
[\Pi_{f_1,f_2},\Pi_{f_1,f_2}]_S=[X_{f_1}^h \wedge X_{f_2}^v,
X_{f_1}^h \wedge X_{f_2}^v]_S= 2 \Pi_{f_1,f_2} \wedge [X_{f_1}^h,
X_{f_2}^v],
$$
and then, from Lemma \ref{lift}, we deduce that
$[\Pi_{f_1,f_2},\Pi_{f_1,f_2}]_S=0$.
\end{proof}

\begin{proposition}
Let $(M_1, \Pi_1)$ and $(M_2, \Pi_2)$ be two Poisson manifolds and
let $f_i,\mu_i\in C^\infty(M_i)$, $i= 1,2$. Let $\Pi_{f_1,f_2}$ be
the Poisson tensor given in the proposition above. If $\mu_1$ and
$\mu_2$ are Casimir functions, then the bivector field
$$
\Lambda= \mu_2^v\Pi_1^h + \mu_1^h\Pi_2^v+\mu_1^h\mu_2^v
\Pi_{f_1,f_2}
$$
is a Poisson tensor on $M_1\times M_2$.
\end{proposition}
\begin{proof}
For any $\varphi_i,\psi_i\in C^\infty(M_i)$, $i=1,2$, using Lemma
\ref{lift} we can easily verify that
$$
\Pi_{f_1,f_2}(d\varphi_1^h,d\psi_1^h)=0,\qquad
\Pi_{f_1,f_2}(d\varphi_2^v,d\psi_2^v)=0 $$ and
$$
\Pi_{f_1,f_2}(d\varphi_1^h,d\varphi_2^v)=\{f_1,\varphi_1\}_1^h\{f_2,\varphi_2\}_2^v.
$$
If for $\varphi,\psi\in C^{\infty}(M_1\times M_2)$ we put
$\{\varphi,\psi\}=\Lambda(d\varphi,d\psi)$, we deduce using the
identities above that
$$
\{\varphi_1^h,\psi_1^h\}=\mu_2^v \{\varphi_1,\psi_1\}_1^h, \quad
\{\varphi_2^v,\psi_2^v\}=\mu_1^h\{\varphi_2,\psi_2\}_2^v
$$
and
$$
\{\varphi_1^h,\varphi_2^v\}=\mu_1^h\mu_2^v\{f_1,\varphi_1\}_1^h\{f_2,\varphi_2\}_2^v.
$$
To prove that the bivector field $\Lambda$ is a Poisson tensor, we
need to prove that the bracket $\{.,.\}$ satisfies the Jacobi
identity. Let $\varphi_i,\phi_i,\psi_i\in C^{\infty}(M_i)$, $i=1,2$.
Using the above identities and the Leibniz identity we get
$$
\{\{\varphi_1^h,\phi_1^h\},\psi_1^h\}
=-\mu_1^h\mu_2^v\{f_2,\mu_2\}_2^v\{\varphi_1,\phi_1\}_1^h\{f_1,\psi_1\}_1^h
+(\mu_2^v)^2\{\{\varphi_1,\phi_1\}_1,\psi_1\}_1^h
$$
and since $\Pi_1$ is Poisson, taking the cyclic sum
$\oint_{\varphi_1,\phi_1,\psi_1}$, we get
\begin{align*}
\oint_{\varphi_1,\phi_1,\psi_1}\{\{\varphi_1^h,\phi_1^h\},\psi_1^h\}
=-\mu_1^h\mu_2^v\{f_2,\mu_2\}_2^v\oint_{\varphi_1,\phi_1,\psi_1}\{\varphi_1,\phi_1\}_1^h\{f_1,\psi_1\}_1^h,
\end{align*}
and using the same arguments we also get
\begin{align*}
\oint_{\varphi_2,\phi_2,\psi_2}\{\{\varphi_2^v,\phi_2^v\},\psi_2^v\}
=\mu_1^h\mu_2^v\{f_1,\mu_1\}_1^h\oint_{\varphi_2,\phi_2,\psi_2}\{\varphi_2,\phi_2\}_2^v\{f_2,\psi_2\}_2^v,
\end{align*}
\begin{align*}
\displaystyle\oint_{\varphi_1,\phi_1,\psi_2}\{\{\varphi_1^h,\phi_1^h\},\psi_2^v\}
&=(\mu_2^v)^2\{f_2,\psi_2\}_2^v[\{f_1,\phi_1\}_1\{\mu_1,\varphi_1\}_1-\{f_1,\varphi_1\}_1\{\mu_1,\phi_1\}_1]^h\\
&+\mu_1^h\{\varphi_1,\phi_1\}_1^h\{\mu_2,\psi_2\}_2^v
\end{align*}
and
\begin{align*}
\displaystyle\oint_{\varphi_2,\phi_2,\psi_1}\{\{\varphi_2^v,\phi_2^v\},\psi_1^h\}
&=(\mu_1^h)^2\{f_1,\psi_1\}_1^h[\{f_2,\varphi_2\}_2\{\mu_2,\phi_2\}_2-\{f_2,\phi_2\}_2\{\mu_2,\varphi_2\}_2]^v\\
&+\mu_2^v\{\varphi_2,\phi_2\}_2^v\{\mu_1,\psi_1\}_1^h.
\end{align*}
Now, $\mu_1$ and $\mu_2$ being Casimir functions, we deduce that
\begin{align*}
\oint_{\varphi_1,\phi_1,\psi_2}\{\{\varphi_1^h,\phi_1^h\},\psi_1^h\}&
=\oint_{\varphi_1,\phi_1,\psi_2}\{\{\varphi_2^v,\phi_2^v\},\psi_2^v\}
=\oint_{
\varphi_1,\phi_1,\psi_2}\{\{\varphi_1^h,\phi_1^h\},\psi_2^v\}\\
&=\oint_{\varphi_1,\phi_1,\psi_2}\{\{\varphi_2^v,\phi_2^v\},\psi_1^h\}=0.
\end{align*}
\end{proof}

\section{Warped bivector fields on warped products}

In this section, we define the contravariant warped product in the
same way the covariant warped product was defined in \cite{bishop}.
On a contravariant warped product equipped with a warped bivector
field, we compute the Levi-Civita contravariant connection and the
associated curvatures. Several proofs contain standard but long
computations, and hence will be omitted.

\subsection{The Levi-Civita contravariant connection on a
warped product manifold equipped with a warped bivector field}

Let $(M_1,\tilde{g}_1)$ and $(M_2,\tilde{g}_2)$ be two
pseudo-Riemannian manifolds and let $g_1$ and $g_2$ be the cometrics
of $\tilde{g}_1$ and $\tilde{g}_2$ respectively. Let $f$ be a
positive smooth function on $M_1$. The contravariant metric
$g^f=g_1^h+f^hg_2^v$ on the product manifold $M_1\times M_2$ is
characterized by the following identities
\begin{equation}  \label{warped metric}
g^{f}(\alpha_1^h,\beta_1^h)=g_1(\alpha_1,\beta_1)^h, \hskip 0.2cm
 g^{f}(\alpha_2^v,\beta_2^v)=f^h g_2(\alpha_2,\beta_2)^v, \hskip 0.2cm g^{f}(\alpha_1^h,\alpha_2^v)=0,
\end{equation}
for any $\alpha_i,\beta_i\in\Gamma(T^*M_i)$, $i=1,2$. We call
$(M_1\times M_2,g^f)$ the contravariant warped product of
$(M_1,\tilde{g}_1)$ and $(M_2,\tilde{g}_2)$. The following lemma
shows that the contravariant tensor $g^f$ is nothing else than the
cometric of the warped metric $\tilde{g}_{\frac{1}{f}}$.

\begin{lemma}\label{shg}  Let $\alpha_i,\beta_i\in\Gamma(T^*M_i)$ and $X_i\in
\Gamma(TM_i)$, $i=1,2$. Let $\alpha=\alpha_1^h+\alpha_2^v$ and
$X=X_1^h+X_2^v$. We have
\begin{enumerate}
\item $\sharp_{g^{f}}(\alpha)=[\sharp_{g_1}(\alpha_1)]^h+f^h[%
\sharp_{g_2}(\alpha_2)]^v$,
\item $\sharp_{g^{f}}^{-1}(X)=[\sharp_{g_1}^{-1}(X_1)]^h+\frac{1}{f^h}%
[\sharp_{g_2}^{-1}(X_2)]^v$,
\item
$\widetilde{g^f}(X,X)=\tilde{g}_1(X_1,X_1)^h+\frac{1}{f^h}\tilde{g}_2(X_2,X_2)^v$,
where $\widetilde{g^f}$ is the metric on $M_1\times M_2$ whose
cometric is $g^f$.
\end{enumerate}
\end{lemma}
\begin{proof}
The proof of the first assertion is the same as that of the first
assertion in Proposition \ref{bracket}. For the second assertion, it
suffices to put $\sharp_{g_1}(\alpha_1)=X_1$ and
$\sharp_{g_2}(\alpha_2)=X_2$ in 1. The third assertion follows from
the assertions 1. and 2.
\end{proof}

Let $\Pi_i$ be a bivector field on $M_i$, $i=1,2$, and let $\mu$ be
a smooth function on $M_1$. Using the Koszul formula (\ref{koszul}),
let us compute the Levi-Civita contravariant connection
$\mathcal{D}$, associated with the pair $(g^f,\Pi^{^\mu})$, in terms
of the Levi-Civita connections $\mathcal{D}^1$ and $\mathcal{D}^2$
associated with the pairs $(g_1,\Pi_1)$ and $(g_2,\Pi_2)$
respectively.

\begin{proposition}\label{connection}
For any $\alpha_i,\beta_i\in\Gamma(T^*M_i)$, $i=1,2$, we have
\begin{enumerate}
\item $\mathcal{D}_{\alpha_1^h}\beta_1^h=(\mathcal{D}^1_{\alpha_1}\beta_1)^h$,
\item $\mathcal{D}_{\alpha_2^v}\beta_2^v=\mu^h(\mathcal{D}
^2_{\alpha_2}\beta_2)^v+\frac{1}{2}\Pi_2(\alpha_2,\beta_2)^v(d\mu)^h
+\frac{1}{2}g_2(\alpha_2,\beta_2)^v(J_1d{f})^h$,
\item $\mathcal{D}_{\alpha_1^h}\beta_2^v=-\frac{1}{2f^h}
\left[g_1(J_1df,\alpha_1)^h\beta_2^v+g_1(d\mu,\alpha_1)^h
(J_2\beta_2)^v\right]$,

\item $\mathcal{D}_{\beta_2^v}\alpha_1^h=\mathcal{D}%
_{\alpha_1^h}\beta_2^v$,
\end{enumerate}

\end{proposition}

\begin{proof}
Let $\alpha_i,\beta_i, \gamma_i\in \Gamma(T^*M_i)$, $i=1,2$. For any
$(i,l),(j,l),(k,l)\in \{(1,h),(2,v)\}$, we have
\begin{equation}\label{koszulijk}
\begin{array}{ll}
2{g^{f}}(\mathcal{D}_{\alpha_i^l}\beta_j^l,\gamma_k^l)\!\!\!\! & =
\sharp_{\Pi^{^\mu}}(\alpha_i^l).{g^{f}}(\beta_j^l,\gamma_k^l)+\sharp_{\Pi^{^\mu}}(\beta_j^l).{g^f}(\alpha_i^l,
\gamma_k^l)-\sharp_{\Pi^{^\mu}}(\gamma_k^l).{g^f}(\alpha_i^l,\beta_j^l)\\
&+{g^{f}}([\alpha_i^l,\beta_j^l]_{\Pi^{^\mu}},\gamma_k^l)+{g^{f}}([\gamma_k^l,\alpha_i^l]_{\Pi^{^\mu}},\beta_j^l)
+{g^{f}}([\gamma_k^l,\beta_j^l]_{\Pi^{^\mu}}, \alpha_i^l).
\end{array}
\end{equation}
1. Taking $(i,l)=(j,l)=(k,l)=(1,h)$ in this formula, using Formula
(\ref{warped metric}) and Proposition \ref{bracket}, we get
$$
2g^{f}(\mathcal{D}_{\alpha_1^h}\beta_1^h,\gamma_1^h)=2(g_1(\mathcal{D}^1_{\alpha_1}\beta_1,\gamma_1))^h,
$$
and using (\ref{warped metric}) again, we get
$$
g^{f}(\mathcal{D}_{\alpha_1^h}\beta_1^h,\gamma_1^h)=g^f((\mathcal{D}^1_{\alpha_1}\beta_1)^h,\gamma_1^h).
$$
Similarly, taking $(i,l)=(j,l)=(1,h)$ and $(k,l)=(2,v)$, we get $
{g^{f}}(\mathcal{D}_{\alpha_1^h}\beta_1^h,\gamma_2^v) =0$ and then
$$
{g^{f}}(\mathcal{D}_{\alpha_1^h}\beta_1^h,\gamma_2^v)
=g^f((\mathcal{D}^1_{\alpha_1}\beta_1)^h,\gamma_2^v).
$$
The result follows. \smallskip\\
2. Taking $(i,l)=(j,l)=(2,v)$ and $(k,l)=(1,h)$ in
(\ref{koszulijk}), using Formula (\ref{warped metric}) and
Proposition \ref{bracket}, we get
\begin{align*}
g^{f}(\mathcal{D}_{\alpha_2^v}\beta_2^v,\gamma_1^h)=&\frac{1}{2}
\left\{g^{f}([\alpha_2^v,\beta_2^v]_{\Pi^{^\mu}},\gamma_1^h)
-\sharp_{\Pi^{^\mu}}(\gamma_1^h).{g^{f}}(\alpha_2^v,\beta_2^v)\right\}\\
=&\frac{1}{2}
\left\{\Pi_2(\alpha_2,\beta_2)^v g_1(d\mu,\gamma_1)^h-g_2(\alpha_2,\beta_2)^v\Pi_1(\gamma_1,df)^h\right\}\\
=&g^f(\frac{1}{2} \left\{\Pi_2(\alpha_2,\beta_2)^v (d\mu)^h+
g_2(\alpha_2,\beta_2)^v(J_1df)^h \right\} ,\gamma_1^h),
\end{align*}
and similarly, taking $(i,l)=(j,l)=(k,l)=(2,v)$ in
(\ref{koszulijk}), we get
$$
g^{f}(\mathcal{D}_{\alpha_2^v}\beta_2^v,\gamma_2^v)
=\mu^hg^{f}((\mathcal{D}^2_{\alpha_2}\beta_2)^v,\gamma_2^v).
$$
3. This is analogous to the proofs of 1. and 2.\smallskip\\
4. Since $\mathcal{D}$ is torsion-free we have
$\mathcal{D}_{\beta_2^v}\alpha_1^h=\mathcal{D}_{\alpha_1^h}\beta_2^v
+[\alpha_1^h,\beta_2^v]_{\Pi^{^\mu}}$. By Proposition \ref{bracket},
we have $[\alpha_1^h,\beta_2^v]_{\Pi^{^\mu}}=0$.
\end{proof}

\begin{proposition}\label{contra Poisson}
Let $\alpha_i,\beta_i,\gamma_i\in\Gamma(T^*M_i)$, $i=1,2$. We have
\begin{align*}
1.~
\mathcal{D}\Pi^{^\mu}(\alpha_1^h,\beta_1^h,\gamma_1^h)=
&\left[\mathcal{D}^1\Pi_1(\alpha_1,\beta_1,\gamma_1)\right]^h, \\
2.~
\mathcal{D}\Pi^{^\mu}(\alpha_2^v,\beta_2^v,\gamma_2^v)=
&(\mu^2)^h\left[\mathcal{D}^2\Pi_{2}(\alpha_2,\beta_2,\gamma_2)\right]^v, \\
3.~
\mathcal{D}\Pi^{^\mu}(\alpha_1^h,\beta_1^h,\gamma_2^v)=
&\mathcal{D}\Pi^{^\mu}(\alpha_1^h,\beta_2^v,\gamma_1^h)=
\mathcal{D}\Pi^{^\mu}(\alpha_2^v,
\beta_1^h,\gamma_1^h)=0, \\
4.
~\mathcal{D}\Pi^{^\mu}(\alpha_1^h,\beta_2^v,\gamma_2^v)=&\frac{\mu^h}{f^h}
g_1(J_1df,\alpha_1)^h\Pi_2(\beta_2,\gamma_2)^v, \\
5.~ \mathcal{D}\Pi^{^\mu}(\alpha_2^v,\beta_1^h,\gamma_2^v)=
&\frac{\mu^h}{2f^h}\left[\Pi_2(\alpha_2,\gamma_2)^vg_1(J_1df,\beta_1)^h
-\Pi_2(\alpha_2,J_2\gamma_2)^v g_1(d\mu,\beta_1)^h\right] \\
&+\frac{1}{2}\left[g_2(\alpha_2,\gamma_2)^v\Pi_1(J_1df,\beta_1)^h
-g_2(\alpha_2,J_2\gamma_2)^v\Pi_1(d\mu,\beta_1)^h\right],\\
6.~
\mathcal{D}\Pi^{^\mu}(\alpha_2^v,\beta_2^v,\gamma_1^h)=
&\frac{\mu^h}{2f^h}\left[\Pi_2(\alpha_2,J_2\beta_2)^v
g_1(d\mu,\gamma_1)^h-\Pi_2(\alpha_2,\beta_2)^vg_1(J_1df,\gamma_1)^h\right] \\
&+\frac{1}{2}\left[g_2(\alpha_2,J_2\beta_2)^v\Pi_1(d\mu,\gamma_1)^h
-g_2(\alpha_2,\beta_2)^v\Pi_1(J_1d{f},\gamma_1)^h\right].
\end{align*}
\end{proposition}

\begin{proof}
The proof uses (\ref{connection multi}), Proposition \ref{bracket}
an Proposition \ref{connection}.
\end{proof}

\subsection{The curvatures of the Levi-Civita contravariant connection}

In the following proposition, we express the curvature $\mathcal{R}$
of the contravariant connection $\mathcal{D}$ in terms of the
warping functions $f,\mu$ and the curvatures $\mathcal{R}_1$ and
$\mathcal{R}_2$ of $\mathcal{D}^1$ and $\mathcal{D}^2$ respectively.

\begin{proposition}\label{curvature}
Let $\alpha_i,\beta_i,\gamma_i\in\Gamma(T^*M_i)$, $i=1,2$, and let
$\gamma=\gamma_1^h+\gamma_2^v$. We have
\begin{align*}
1.~\mathcal{R}(\alpha_1^h,\beta_1^h)\gamma=\left[\mathcal{R}_1(\alpha_1,\beta_1)
\gamma_1\right]^h+\left[g_1(\mathcal{D}_{\beta_1}^1\dfrac{d\mu}{2{f}},\alpha_1)
-g_1(\mathcal{D}_{\alpha_1}^1\dfrac{d\mu}{2{f}},\beta_1)\right]^h(J_2\gamma_2)^v,
\end{align*}
\begin{align*}
2.~\mathcal{R}(\alpha_1^h,\beta_2^v)\gamma_1^h
=&\frac{1}{4(f^h)^2}\left[g_1(J_1df,\alpha_1)g_1(J_1df,\gamma_1)
-2f g_1(\mathcal{D}^1_{\alpha_1}(J_1df),\gamma_1)\right]^h\beta_2^v\\
&+\frac{1}{4(f^h)^2}\left[g_1(d\mu,\alpha_1)g_1(J_1df,\gamma_1)
+\!g_1(J_1df,\alpha_1)g_1(d\mu,\gamma_1)\right]^h\!(J_2\beta_2)^v\\
& -g_1(\mathcal{D}^1_{\alpha_1}\frac{d\mu}{2f},\gamma_1)^h(J_2
\beta_2)^v+\!\!\frac{1}{4(f^h)^2}\left[g_1(d\mu,\alpha_1)g_1(d\mu,\gamma_1)\right]^h\!(J_2^2\beta_2)^v,
\end{align*}
\begin{align*}
3.~&\mathcal{R}(\alpha_1^h,\beta_2^v)\gamma_2^v =
\frac{1}{4{f}^h}\left[g_1(d\mu,\alpha_1)^h\Pi_2(\beta_2,J_2\gamma_2)^v+
g_1(J_1df,\alpha_1)^h\Pi_2(\beta_2,\gamma_2)^v\right](d\mu)^h\\
& +\frac{1}{4{f}^h}\left[g_1(d\mu,\alpha_1)^h
g_2(\beta_2,J_2\gamma_2)^v+g_1(J_1df,\alpha_1)^hg_2(\beta_2,\gamma_2)^v\right](J_1d{f})^h\\
&+\frac{\mu^h}{2f^h}\left[g_1(d\mu,\alpha_1)^h\left(\mathcal{D}^2_{\beta_2}(J_2
\gamma_2)- J_2\mathcal{D}^2_{\beta_2}\gamma_2\right)^v
-g_1(J_1df,\alpha_1)^h(\mathcal{D}^2_{\beta_2}\gamma_2)^v\right]\\
&  +\frac{1}{2}\Pi_2(\beta_2,\gamma_2)^v(\mathcal{D}^1_{\alpha_1}d\mu)^h+%
\frac{1}{2}g_2(\beta_2,\gamma_2)^v(\mathcal{D}^1_{\alpha_1}(J_1df))^h
-\Pi_1(d\mu,\alpha_1)^h(\mathcal{D}^2_{\beta_2}\gamma_2)^v,
\end{align*}
\begin{align*}
\shoveright{4.~\mathcal{R}(\alpha_2^v,\beta_2^v)\gamma_1^h
=&\frac{1}{2{f^h}}\Pi_2(\alpha_2,\beta_2)^v
\left[g_1(d\mu,\gamma_1)(J_1df)
-g_1(J_1df,\gamma_1)d\mu-\mathcal{D}^1
_{d\mu}\gamma_1\right]^h} \\
 & -\frac{\mu^h}{2f^h}g_1(d\mu,\gamma_1)^h
 \left[\mathcal{D}^2_{\alpha_2}(J_2\beta_2)
 -\mathcal{D}^2_{\beta_2}(J_2\alpha_2)-J_2[
\alpha_2,\beta_2]_{\Pi_2}\right]^v,
\end{align*}
\begin{align*}
5.~\mathcal{R}(\alpha_2^v,\beta_2^v)\gamma_2^v = &
(\mu^2)^h\left[\mathcal{R}_2(\alpha_2,\beta_2)\gamma_2\right]^v\\
&+\frac{\mu^h}{2}\left[\mathcal{D}^2\Pi_2(\alpha_2,\beta_2,\gamma_2)
-\mathcal{D}^2\Pi_2(\beta_2,\alpha_2,\gamma_2) \right]^v(d\mu)^h\\
&+\left(\frac{\|d\mu\|_1^2}{4f}\right)^h\left[J_2\left(\Pi_2(\alpha_2,\gamma_2)\beta_2
-\Pi_2(\beta_2,\gamma_2)\alpha_2+2\Pi_2(\alpha_2,\beta_2)\gamma_2\right)\right]^v\\
&+\left(\frac{\|J_1df\|_1^2}{4f}\right)^h\left[g_2(\alpha_2,\gamma_2)\beta_2
-g_2(\beta_2,\gamma_2)\alpha_2\right]^v\\
&+\left(\frac{g_1(d\mu,J_1df)}{4f}\right)^h\!\!\left[\Pi_2(\alpha_2,\gamma_2)\beta_2
-\Pi_2(\beta_2,\gamma_2)\alpha_2+\!2\Pi_2(\alpha_2,\beta_2)\gamma_2\right]^v\\
&+\left(\frac{g_1(d\mu,J_1df)}{4f}\right)^h\left[J_2\left(g_2(\alpha_2,\gamma_2)\beta_2
-g_2(\beta_2,\gamma_2)\alpha_2\right)\right]^v.
\end{align*}
\end{proposition}

\begin{proof}
Long but straightforward computations using Propositions
\ref{bracket} and \ref{connection}.
\end{proof}

Now, in the following two corollaries, we express the Ricci
curvature $r$ (resp. the scalar curvature $\mathcal{S}$) of the
contravariant connection $\mathcal{D}$ in terms of the warping
functions $f,\mu$ and the Ricci curvatures $r_i$ (resp. the scalar
curvatures $\mathcal{S}_i$) of $\mathcal{D}^i$, $i=1,2$. For the
proof, notice that if we choose $\{dx_1,dx_2,\ldots,dx_{n_1}\}$ to
be a local $g_1$-orthonormal basis of the $1$-forms on an open
$U_1\subseteq M_1$ and $\{dy_1,dy_2,\ldots,dy_{n_2}\}$ to be a local
$g_2$-orthonormal basis of the $1$-forms on an open $U_2\subseteq
M_2$, then
$$
\{dx_1^h,\ldots,dx_{n_1}^h,\frac{1}{\sqrt{{f}^h}}dy_1^v,\ldots,\frac{1}{\sqrt{{f}^h}}dy_{n_2}^v\}
$$
is a local $g^f$-orthonormal basis of the $1$-forms on the open
$U_1\times U_2$ of $M_1\times M_2$.

\begin{corollary}\label{ricci} For any $\alpha_i,\beta_i\in\Gamma(T^*M_i)$, $i=1,2$,
we have
\begin{align*}
1. ~r(\alpha_1^h,\beta_1^h)=& r_1(\alpha_1,\beta_1)^h
+\frac{\left(\|J_2\|_2^2\right)^v}{4(f^h)^2}\left(g_1(d\mu,\alpha_1)g_1(d\mu,\beta_1)\right)^h\\
&+\frac{n_2}{4(f^h)^2}\left[g_1(J_1df,\alpha_1)g_1(J_1df,\beta_1)
+2fg_1(\mathcal{D}_{\alpha_1}^1(J_1df),\beta_1)\right]^h
\end{align*}
where
$$
\|J_2\|_2^2=\sum_{i=1}^{n_2}\|J_2dy_i\|_2^2=\sum_{i=1}^{n_2}g_2(J_2dy_i,J_2dy_i)=-trac_{g_2}(J_2\circ J_2),
$$
\begin{align*}
2.~r(\alpha_1^h,\beta_2^v)=&\frac{\mu^h}{2f^h}g_1(d\mu,\alpha_1)^h
\left(tr_{g_2}(\alpha_2\mapsto (\mathcal{D}^2_{\alpha_2}J)(\beta_2))\right)^v \\
&+\left(\Pi_1(d\mu,\alpha_1)+\frac{\mu}{2f}g_1(J_1df,\alpha_1)\right)^h
\left(tr_{g_2}(\mathcal{D}^2\beta_2)\right)^v,
\end{align*}
where
$$
tr_{g_2}(\alpha_2\mapsto (\mathcal{D}^2_{\alpha_2}J)(\beta_2))=\left[\sum_{1=1}^{n_2}g_2(J_2\mathcal{D}_{dy_i}^2\beta_2
-\mathcal{D}_{dy_i}^2(J_2\beta_2),dy_i)\right],
$$
and
$$
div\beta_2=tr_{g_2}(\mathcal{D}^2\beta_2)=\left[\sum_{1=1}^{n_2}g_2(\mathcal{D}_{dy_i}^2\beta_2,dy_i)\right],
$$
\begin{align*}
3.~r(\alpha_2^v,\beta_2^v)=&(\mu^2)^h r_2(\alpha_2,\beta_2)^v
-\left(\frac{\|d\mu\|_1^2}{2f}\right)^h\Pi_2(\alpha_2,J_2\beta_2)^v\\
&-\left(\frac{(n_2-2)\|J_1df\|_1^2}{4f}\right)^hg_2(\alpha_2,\beta_2)^v
-\left(\frac{n_2g_1(d\mu,J_1df)}{4f}\right)^h\Pi_2(\alpha_2,\beta_2)^v \\
&+\frac{1}{2}(\rhd_{\Pi_1}(\mu))^h\Pi_2(\alpha_2,\beta_2)^v
+\frac{1}{2}(\lhd_{\Pi_1}(f))^hg_2(\alpha_2,\beta_2)^v.
\end{align*}

\end{corollary}

\begin{proof}
It is a consequence of the foregoing proposition.
\end{proof}

\begin{corollary}\label{scalar}
We have
\begin{align*}
\mathcal{S}=&\mathcal{S}_1^h+\left(\frac{\mu^2}{f}\right)^h
\mathcal{S}_2^v-\left(\frac{\|d\mu\|_1^2}{4f^2}\right)^h\left(\|J_2\|_2^2\right)^v\\
&-\left(\frac{n_2(n_2-3)\|J_1df\|_1^2}{4f^2}-\frac{n_2}{f}\lhd_{\Pi_1}(f)\right)^h.
\end{align*}
\end{corollary}

\begin{proof}
The statement follows directly from Corollary \ref{ricci}.
\end{proof}

Finally, to end this section, in the following corollary we express
the sectional contravariant curvature $\mathcal{K}$ of the
contravariant connection $\mathcal{D}$ in terms of the warping
functions $f,\mu$ and the sectional contravariant curvatures
$\mathcal{K}_1$ and $\mathcal{K}_2$ of $\mathcal{D}^1$ and
$\mathcal{D}^2$ respectively.

\begin{corollary}\label{sectional}
For any $\alpha_i,\beta_i \in \Gamma(T^*M_i)$, $i=1,2$, we have
{\small\begin{align*}
1.~\mathcal{K}(\alpha_1^h,\beta_1^h)&=\mathcal{K}_1(\alpha_1,\beta_1)^h \\
2.~\mathcal{K}(\alpha_1^h,\beta_2^v)&=\left(\frac{g_1^2(d\mu,\alpha_1)}{4f^2\|\alpha_1\|^2_1}
\right)^h\left(\frac{\|J_2\beta_2\|^2_2}{\|\beta_2\|^2_2}\right)^v
+\left(\frac{g_1^2(J_1df,\alpha_1)}{4f^2\|\alpha_1\|^2_1}
+\frac{g_1(\mathcal{D}^1_{\alpha_1}J_1df,\alpha_1)}{2f\|\alpha_1\|_1^2}\right)^h, \\
3.~\mathcal{K}(\alpha_2^v,\beta_2^v)&=\left(\frac{\mu^2}{f}\right)^h\mathcal{K}_2(\alpha_2,\beta_2)^v
-\left(\frac{\|d\mu\|_1^2}{4{f}^2}\right)^h\left(\frac{3\Pi_2^2(\alpha_2,\beta_2)}{\|\alpha_2\|^2_2\|\beta_2\|^2_2
-g_2^2(\alpha_2,\beta_2)}\right)^v\\
&-\left(\frac{\|J_1df\|_1^2}{4{f}^2}\right)^h
+\left(\dfrac{g_1(d\mu,J_1df)}{2f^2}\right)^h\left(\frac{\Pi_2(\alpha_2,\beta_2)
g_2(\alpha_2,\beta_2)}{\|\alpha_2\|^2_2\|\beta_2\|^2_2
-g_2^2(\alpha_2,\beta_2)}\right)^v.
\end{align*}}
\end{corollary}

\begin{proof}
This is a  consequence of Proposition \ref{curvature}.
\end{proof}

\subsection{Geometric consequences}

Finally, we will conclude with some geometric consequences.

\begin{theorem}\label{pseudo-Riemann poisson}
If $f$ is a Casimir function and $\mu$ a nonzero essentially
constant function, the triple $(M_1\times M_2,g^f,\Pi^{^\mu})$ is a
pseudo-Riemannian Poisson manifold if and only if $(M_1,g_1,\Pi_1)$
and $(M_2,g_2, \Pi_2)$ are.
\end{theorem}

\begin{proof}
First, observe that $f$ is a Casimir function if and only if
$J_1df=0$. Now, by Proposition \ref{contra Poisson}, we have
\begin{eqnarray*}
\mathcal{D}\Pi^{^\mu}(\alpha_1^h,\beta_1^h,\gamma_1^h)
&=&[\mathcal{D}^1\Pi_1(\alpha_1,\beta_1,\gamma_1)]^h,\\
\mathcal{D}\Pi^{^\mu}(\alpha_2^v,\beta_2^v,\gamma_2^v)
&=&(\mu^2)^h[\mathcal{D}^2\Pi_{2}(\alpha_2,\beta_2,\gamma_2)]^v,
\end{eqnarray*}
$$
\mathcal{D}\Pi^{^\mu}(\alpha_1^h,\beta_1^h,\gamma_2^v)=
\mathcal{D}\Pi^{^\mu}(\alpha_1^h,\beta_2^v,\gamma_1^h)=
\mathcal{D}\Pi^{^\mu}(\alpha_2^v,\beta_1^h,\gamma_1^h)=0
$$
and, since $\mu$ is an essentially constant function ($d\mu=0$) and
$f$ is Casimir,
$$
\mathcal{D}\Pi^{^\mu}(\alpha_1^h,\beta_2^v,\gamma_2^v)=
\mathcal{D}\Pi^{^\mu}(\alpha_2^v,\beta_1^h,\gamma_2^v)=
\mathcal{D}\Pi^{^\mu}(\alpha_2^v,\beta_2^v,\gamma_1^h)=0.
$$
This shows the equivalence.
\end{proof}

\begin{theorem}
Under the same assumptions as in Theorem \ref{pseudo-Riemann
poisson}, the triple $(M_1\times M_2,g^f,\Pi^{^\mu})$ is locally
symmetric if and only if $(M_1,g_1,\Pi_1)$ and $(M_2,g_2,\Pi_2)$
are.
\end{theorem}

\begin{proof}
Since $f$ is Casimir and $\mu$ is essentially constant, by
Proposition \ref{curvature}, for any
$\alpha_i,\beta_i,\gamma_i\in\Gamma(T^*M_i)$, $i=1,2$, we have
$$
\mathcal{R}(\alpha_1^h,\beta_1^h)\gamma_1^h=\left[\mathcal{R}_1(\alpha_1,\beta_1)\gamma_1\right]^h,\qquad
\mathcal{R}(\alpha_2^v,\beta_2^v)\gamma_2^v=(\mu^2)^h[\mathcal{R}_2(\alpha_2,\beta_2)\gamma_2]^v
$$
and
\begin{align*}
\mathcal{R}(\alpha_1^h,\beta_1^h)\gamma_2^v=
\mathcal{R}(\alpha_1^h,\beta_2^v)\gamma_1^h&=
\mathcal{R}(\alpha_2^v,\beta_1^h)\gamma_1^h\\
&=\mathcal{R}(\alpha_1^h,\beta_2^v)\gamma_2^v=
\mathcal{R}(\alpha_2^v,\beta_1^h)\gamma_2^v=
\mathcal{R}(\alpha_2^v,\beta_1^v)\gamma_1^h=0,
\end{align*}
and by Proposition \ref{connection}, for any
$\alpha_i,\beta_i\in\Gamma(T^*M_i)$, $i=1,2$, we have
$$
\mathcal{D}_{\alpha_1^h}\beta_1^h=
(\mathcal{D}^1_{\alpha_1}\beta_1)^h, \qquad
\mathcal{D}_{\alpha_2^v}\beta_2^v=\mu^h
(\mathcal{D}^2_{\alpha_2}\beta_2)^v,
$$
and
$$
\mathcal{D}_{\alpha_1^h}\beta_2^v=\mathcal{D}_{\alpha_2^v}\beta_1^h=0.
$$
Therefore, by (\ref{locally-sym}), for
$\alpha_i,\beta_i,\gamma_i,\delta_i\in\Gamma(T^*M_i)$, $i=1,2$, if
we set $\alpha=\alpha_1^h+\alpha_2^v$, $\beta=\beta_1^h+\beta_2^v$,
$\gamma=\gamma_1^h+\gamma_2^v$ and $\delta=\delta_1^h+\delta_2^v$,
we have
$$
(\mathcal{D}_{\alpha}\mathcal{R})(\beta,\gamma)\delta=
\left[(\mathcal{D}^1_{\alpha_1}\mathcal{R}_1)(\beta_1,\gamma_1)\delta_1\right]^h+
({\mu}^3)^h\left[(\mathcal{D}^2_{\alpha_2}\mathcal{R}_2)(\beta_2,\gamma_2)\delta_2\right]^v.
$$
Hence $\mathcal{D}\mathcal{R}=0$ if and only if
$\mathcal{D}^1\mathcal{R}_1=\mathcal{D}^2\mathcal{R}_2=0$.
\end{proof}

\begin{theorem}\label{flat}
Under the same assumptions as in Theorem \ref{pseudo-Riemann
poisson}, the triple $(M_1\times M_2,g^f,\Pi^{^\mu})$ is flat if and
only if $(M_1,g_1,\Pi_1)$ and $(M_2,g_2,\Pi_2)$ are flat.
\end{theorem}

\begin{proof}
In the proof of the theorem above we see that $\mathcal {R}=0$ if
and only if $\mathcal{R}_1=\mathcal{R}_2=0$.
\end{proof}

\begin{theorem}
Under the same assumptions as in Theorem \ref{pseudo-Riemann
poisson}, the triple $(M_1\times M_2,g^f,\Pi^{^\mu})$ is Ricci flat
if and only if $(M_1,g_1,\Pi_1)$ and $(M_2,g_2,\Pi_2)$ are Ricci
flat.
\end{theorem}
\begin{proof}
By Corollary \ref{ricci}, for any
$\alpha_i,\beta_i\in\Gamma(T^*M_i)$, $i=1,2$ we have
$$
r(\alpha_1^h,\beta_1^h)=r_1(\alpha_1,\beta_1)^h, \qquad
r(\alpha_2^v,\beta_2^v)=({\mu}^2)^hr_2(\alpha_2,\beta_2)^v
$$
and
$$
r(\alpha_1^h,\beta_2^v)=r(\alpha_2^v,\beta_1^h)=0.
$$
Therefore, $r=0$ if and only if $r_1=r_2=0$.
\end{proof}

Also, under the same assumptions, i.e. $f$ Casimir and $\mu$ nonzero
essentially constant, we deduce from Corollary \ref{sectional} that
$\mathcal{K}=0$ if and only if $\mathcal{K}_1=\mathcal{K}_2=0$.

\begin{proposition}
Assume that ${\mu}$ is a nonzero essentially constant function. If
$(M_1\times M_2,g^f,\Pi^{^\mu})$ has a constant sectional curvature
$k$, then both $(M_1,g_1,\Pi_1)$ and $(M_2,g_2,\Pi_2)$ have a
constant sectional curvature,
$$
\mathcal{K}_1=k \quad \textrm{ and } \quad
\mathcal{K}_2=\frac{f}{\mu^2}k+\frac{1}{4\mu^2f}\|J_1df\|_1^2.
$$
Furthermore, if $f$ is Casimir then it is constant.
\end{proposition}

\begin{proof}
By 1. of Corollary \ref{sectional}, for any
$\alpha_1,\beta_1\in\Gamma(T^*M_1)$, we have
$$
\mathcal{K}_1(\alpha_1,\beta_1)^h=\mathcal{K}(\alpha_1^h,\beta_1^h)=k,
$$
hence $\mathcal{K}_1=k$. Since $\mu$ is essentially constant, by 3.
of the same corollary, for any $\alpha_2,\beta_2\in\Gamma(T^*M_2)$,
we have
$$
\mathcal{K}(\alpha_2^v,\beta_2^v)=
\left(\frac{\mu^2}{f}\right)^h\mathcal{K}_2(\alpha_2,\beta_2)^v-\left(\frac{\|J_1df\|_1^2}{4f^2}\right)^h,
$$
hence
$$
\mathcal{K}_2(\alpha_2,\beta_2)^v=\left(\frac{f}{\mu^2}k+\frac{\|J_1df\|_1^2}{4\mu^2f}
\right)^h,
$$
and the proposition follows.
\end{proof}

\begin{theorem}
Under the same assumptions as in Theorem \ref{pseudo-Riemann
poisson}, the triple $(M_1\times M_2,g^f,\Pi^{^\mu})$ is metaflat if
and only if $(M_1,g_1,\Pi_1)$ and $(M_2,g_2,\Pi_2)$ are metaflat.
\end{theorem}
\begin{proof}
Since $f$ is Casimir and $\mu$ is essentially constant, by
Proposition \ref{connection} we see that the contravariant
Levi-Civita connection $\mathcal{D}$ is nothing else than the
connection $\mathcal{D}^\mu$ defined in \S\,\ref{wgPb}, therefore,
the generalized pre-Poisson bracket associated with $\mathcal{D}$ is
precisely the bracket $\{.,.\}_\mu$ associated with
$\mathcal{D}^\mu$. Now, by Theorem \ref{flat} and since the
vanishing of the metacurvature is equivalent to the graded Jacobi
identity $(\ref{gen bracket jacobi})$, one can deduce from Corollary
\ref{generalized Poisson} that $\mathcal{D}$ is metaflat if and only
if $\mathcal{D}^1$ and $\mathcal{D}^2$ are metaflat.
\end{proof}

%For acknowledgements section, please don't number the section, please begin it with \section*{Acknowledgements}

% You may incorporate your references as follows in your main tex file.
% Using BibTex is not recommended but can be handled.

\medskip
% The data information below will be filled by AIMS editorial staff
Received January 2013; revised August 2014. \medskip


\begin{thebibliography}{99}

\bibitem{Beem2} (MR662851)
\newblock J. K. Beem, P. E. Ehrlich and Th. G. Powell,
\newblock {Warped product manifolds in relativity},
\newblock \emph{Selected Studies: Physics-astrophysics, mathematics, history of science}, pp. 41--56, North-Holland, Amesterdam-New York, 1982.

\bibitem{bishop} (MR0251664) [10.1090/S0002-9947-1969-0251664-4]
\newblock R. L. Bishop and B. O'Neill,
\newblock \doititle{Manifolds of negative curvature},
\newblock \emph{Trans. Amer. Math. Soc.,} \textbf{145} (1969), 1--49.

\bibitem{Bou1} (MR1868950) [10.1016/S0764-4442(01)02132-2]
\newblock M. Boucetta,
\newblock \doititle{Compatibilit\'e des structures pseudo-riemanniennes et des structures de Poisson},
\newblock \emph{C. R. Acad. Sci. Paris,} \textbf{333} (2001), 763--768.

\bibitem{Bou2} (MR2053915) [10.1016/j.difgeo.2003.10.013]
\newblock M. Boucetta,
\newblock \doititle{Poisson manifolds with compatible pseudo-metric and pseudo-Riemannian Lie algebras},
\newblock \emph{Differential Geometry and its Applications,} \textbf{20} (2004), 279--291.

\bibitem{Jean-Paul Dufour} (MR2178041)
\newblock J.-P. Dufour and N. T. Zung,
\newblock \emph{Poisson Structures and Their Normal Forms},
\newblock Progress in Mathematics, vol. \textbf{242}, Birkh\"auser Verlag, Basel, 2005.

\bibitem{fern} (MR1818181)
\newblock R. L. Fernandes,
\newblock {Connections in Poisson geometry I: Holonomy and invariants},
\newblock \emph{J. Diff. Geom.,} \textbf{54} (2000), 303--365.

\bibitem{Hawkins2} (MR2048556) [10.1007/s00220-004-1036-4]
\newblock E. Hawkins,
\newblock \doititle{Noncommutative rigidity},
\newblock \emph{Commun. Math. Phys.,} \textbf{246} (2004), 211--235.

\bibitem{Hawkins} (MR2362320)
\newblock E. Hawkins,
\newblock {The structure of noncommutative deformations},
\newblock \emph{J. Diff. Geom.,} \textbf{77} (2007), 385--424.

\bibitem{Nas}
\newblock R. Nasri and M. Djaa,
\newblock {Sur la courbure des vari\'{e}t\'{e}s riemanniennes produits},
\newblock \emph{Sciences et Technologie,} \textbf{A-24} (2006), 15--20.

\bibitem{Nas2} (MR2728119)
\newblock R. Nasri and M. Djaa,
\newblock {On the geometry of the product Riemannian manifold with the Poisson structure},
\newblock \emph{International Electronic Journal of Geometry,} \textbf{3}  (2010), 1--14.

\bibitem{ONeil} (MR719023)
\newblock B. O'Neill,
\newblock \emph{Semi-Riemannian Geometry with Applications to Relativity},
\newblock Academic Press, 1983.

\bibitem{Vais} (MR1269545) [10.1007/978-3-0348-8495-2]
\newblock I. Vaisman,
\newblock \emph{Lectures on the Geometry of Poisson Manifolds},
\newblock Progress in Mathematics, vol. \textbf{118}, Birkh\"{a}user Verlag, Basel, 1994.



\end{thebibliography}
\end{document}